\documentclass[10pt, a4]{article}

\usepackage[utf8]{inputenc}
\usepackage[english, czech]{babel}  
\usepackage{amsmath,amscd}
\usepackage{amsthm}
\usepackage{mathtools}
\usepackage{amssymb}
\usepackage{esint}
\usepackage[usenames,dvipsnames,svgnames,table]{xcolor}
\usepackage{float}
\usepackage[]{hyperref}
\usepackage{etoolbox}
\PassOptionsToPackage{dvipsnames}{xcolor}
\usepackage{amsthm}
\usepackage[T1]{fontenc}
\usepackage{babel}
\usepackage{tocloft}
\usepackage{blindtext}
\usepackage{tikz-cd}
\usepackage{enumitem}
 \usepackage{sansmath}
\usepackage{dsfont}
 \usepackage{pgfplots}
\usepackage{dsfont}
\usepackage{bbm}
\usepackage{lipsum}

\newcommand\blfootnote[1]{%
  \begingroup
  \renewcommand\thefootnote{}\footnote{#1}%
  \addtocounter{footnote}{-1}%
  \endgroup
}

\usepackage{titlesec}

 \usepackage[margin=1in]{geometry}
 
 \makeatletter
\newcommand*{\rom}[1]{\expandafter\@slowromancap\romannumeral #1@}
\makeatother

\mathtoolsset{showonlyrefs}

\makeatletter
\DeclareFontFamily{U}{tipa}{}
\DeclareFontShape{U}{tipa}{m}{n}{<->tipa10}{}
\newcommand{\arc@char}{{\usefont{U}{tipa}{m}{n}\symbol{62}}}%

\newcommand{\arc}[1]{\mathpalette\arc@arc{#1}}

\newcommand{\arc@arc}[2]{%
  \sbox0{$\m@th#1#2$}%
  \vbox{
    \hbox{\resizebox{\wd0}{\height}{\arc@char}}
    \nointerlineskip
    \box0
  }%
}
\makeatother

\newcommand{\R}{\mathbb{R}}

\newcommand{\Cc}{\mathcal{C}}
\newcommand{\Ss}{\mathcal{S}}
\newcommand{\dd}{\, \textrm{d}}
\newcommand{\loc}{\textrm{loc}}
\newcommand{\pa}{\partial}		

\newcommand{\dom}{\mathcal{D}}
\newcommand{\ub}{\bar{u}}

 \allowdisplaybreaks

\theoremstyle{plain}
\newtheorem{theorem}{Theorem}
  
\theoremstyle{plain}
\newtheorem{intdef}{Definition}
\newenvironment{definition}[1]{\begin{intdef}[#1]\label{def:#1}}{\end{intdef}}

\theoremstyle{plain}
\newtheorem{prop}{Proposition}

 \newtheorem{corollary}{Corollary}

 \newtheorem{lemma}{Lemma}

\newtheorem*{prop*}{Proposition}

\newtheorem{remark}{Remark}

\newtheorem{example}{Example}

\hypersetup{
pdfauthor={Marcel Braukhoff, Claudia Raithel, Nicola Zamponi},
pdftitle={},
breaklinks=true,
colorlinks=true,
linkcolor=black,
citecolor=black,
urlcolor=blue,
filecolor=blue,
}

\numberwithin{equation}{section} 

\title{Partial H\"{o}lder Regularity for Solutions of a Class of Cross-Diffusion Systems with Entropy Structure}

\author{
Marcel~Braukhoff$^{a}$
\and
Claudia~Raithel$^{b,*}$
\and
Nicola~Zamponi$^{c}$
}

\date{\small{ 
    $^a$ Email: marcel.braukhoff$@$hhu.de\\
$^b$ Technische Universit\"at Wien, Wiedner Hauptstr. 8-10, 1040 Wien, Austria\\
Email: claudia.raithel$@$tuwien.ac.at; *corresponding author\\
$^c$ Technische Universit\"at Wien, Wiedner Hauptstr. 8-10, 1040 Wien, Austria\\
Email: nicola.zamponi$@$asc.tuwien.ac.at}}

\begin{document}

\maketitle

\selectlanguage{english}
\begin{abstract} In this article we show a $C^{0,\alpha}$-partial regularity result for solutions of a certain class of cross-diffusion systems with entropy structure. Under slightly more stringent conditions on the system, we are able to obtain a $C^{1,\alpha}$-partial regularity result. Amongst others, our results yield the partial $C^{1,\alpha}$-regularity of weak solutions of the Maxwell-Stefan system, as well as the partial $C^{1,\alpha}$-regularity of bounded weak solutions of the Shigesada-Kawasaki-Teramoto model. The classical partial regularity theory for nonlinear parabolic systems as developed by Giaquinta and Struwe in the 80s proceeds by Campanato iteration which relies on energy methods. Our analysis here centers around the insight that, in the Campanato iteration strategy, we can replace the use of energy estimates by ``entropy dissipation inequalities'' and the use of the squared $L^2$-distance to measure the distance between functions by the use of the ``relative entropy''.  In order for our strategy to work, it is necessary to regularize the entropy structure of the cross-diffusion system, thereby introducing a new technical tool, which we call the ``glued entropy''.
\end{abstract}

\vspace{.2cm}


\textit{Declarations of interest:} None.


\blfootnote{$\dagger$ 2020 Mathematics Subject Classification: 35B65, 35K65, 35K59, 35Q92, 92D25} 

\blfootnote{$\ddagger$ The authors gratefully acknowledge support from the Austrian Science Fund (FWF), grants P30000, P33010, W1245, and F65. The third author would like to gratefully acknowledge the support of the Alexander von Humboldt Foundation (AvH) and the bilateral Croatian-Austrian Project of the Austrian Agency for International Cooperation in Education and Research (\"OAD), grant HR 19/2020. The authors would also like to gratefully acknowledge the support of the bilateral Czech-Austrian project of the \"OAD, grant CZ 15/2018.}


\section{Introduction}

In this paper we are interested in the partial H\"{o}lder regularity of weak solutions of cross-diffusion systems, which are reaction-diffusion systems with non-diagonal diffusion coefficients. Systems of this type find application in many areas, including in the modelling of gaseous or fluid mixtures \cite{M_1866,S_1871}, the dynamics of competing subpopulations \cite{SKT_79}, or in the study of tumour growth \cite{JB_2002}. Formally, cross-diffusion systems have the form
\begin{align}
\label{cross_diffusion_system}
\partial_t u_i - \sum_{j=1}^n\nabla \cdot A_{ij}(u) \nabla u_j & = f_i(u) \quad  \textrm{in} \quad \Omega \times (0,T);
\end{align}
the components $u_i$, for $i = 1, \ldots , n$, are interpreted as chemical or population densities, the interactions of which are governed by the diffusion coefficients $A_{ij}(u)$ and the reaction terms $f_i(u)$. We assume that $\Omega \subset  \mathbb{R}^d$, for $d \geq 2$, is a bounded smooth domain and $T >0$. In the sequel, we abbreviate \eqref{cross_diffusion_system} as 
\begin{align}
\label{cross_diffusion_system_2}
\partial_t u - \nabla \cdot A(u) \nabla u & = f(u) \quad  \textrm{in} \quad \Omega \times (0,T),
\end{align}
where $u = (u_1, ..., u_n)$ and $(A(u))_{ij} = A_{ij}(u)$.

As the diffusion matrix $A(u)$ is neither assumed to be symmetric nor positive definite, the issue of obtaining \textit{a priori} estimates for solutions of \eqref{cross_diffusion_system} can be rather delicate. In particular, without further insight, the standard energy methods that are classically used to obtain partial H\"older regularity results in the context of nonlinear parabolic (or elliptic) systems are not applicable in the setting of \eqref{cross_diffusion_system}. In order to overcome this difficulty, in this paper we restrict ourselves to the class of cross-diffusion systems with an \textit{entropy structure} --The simplest version of such an entropy structure is when there exists a convex domain $\dom \subseteq \R^n_+$, $\overline{\dom}$ containing the range of $u$, and a convex function $h: \dom \rightarrow \R$ such that $h^{\prime \prime} A : \dom \rightarrow \R^{n \times n}$ satisfies
\begin{align}
\label{main_entropy_condition_non}
\rho \cdot h^{\prime \prime} (y) A(y) \rho \geq \lambda |\rho|^2 \quad \textrm{for some } \lambda >0 \textrm{ and any }  y \in \dom \text{ and } \rho \in \R^n.
\end{align}
This function $h$ is called the \textit{entropy density} and $h^{\prime \prime}$ denotes the $n \times n$-dimensional Hessian.

The presence of an entropy structure is useful in the analysis of cross-diffusion systems because it gives one access to an \textit{entropy dissipation inequality}. In particular, defining the entropy as $\mathcal{H}(u) := \int_{\Omega} h(u) \dd x$, we then have that
\begin{align}
\label{entropy_estimate}
\begin{split}
\partial_t \mathcal{H}[u]  = \int_{\Omega} \partial_t u \cdot h^{\prime}(u) \dd x & = - \int_{\Omega} \nabla u : h^{\prime \prime} (u)A(u)\nabla u \dd x + \int_{\Omega} f(u) \cdot h^{\prime}(u) \dd x\\
& \leq -\lambda \int_{\Omega} |\nabla u|^2 \dd x + \int_{\Omega} f(u) \cdot h^{\prime}(u) \dd x.
\end{split}
\end{align}
The relation \eqref{entropy_estimate} implies that when $ \int_{\Omega} f(u)  \cdot  h^{\prime}(u) \leq 0$  and $ \mathcal{H}(u_0) <\infty$, then $\mathcal{H}$ is a Lyapunov functional for \eqref{cross_diffusion_system}.


In the analysis of cross-diffusion systems with entropy structure, the estimate \eqref{entropy_estimate} often plays a similar role to that of the standard energy estimate in the analysis of parabolic systems that satisfy a positive definiteness condition. This can be seen, e.g., in the existence theory for global weak solutions via the boundedness-by-entropy method. This method was first developed by Burger, Di Francesco, Pietschmann, and Schlake for a 2-species diffusion model with size-exclusion \cite{BFPS_2010} and then generalized to a more broad setting by J\"ungel in \cite{J_2014}. Here, the strategy for obtaining global weak solutions is to do a twofold regularization of \eqref{cross_diffusion_system}: first discretizing the time derivative with a first-order implicit Euler scheme and then adding vanishing viscosity and massive terms. As is seen in \cite[Lemma 5]{J_2014}, the regularized equations can then be solved using a Lax-Milgram argument --one then passes to the limit in the regularization via uniform estimates that are obtained from the entropy dissipation.

Similarly to the replacement of energy estimates by entropy estimates, whenever classical methods in the regularity theory for parabolic systems would call for an estimate on the squared $L^2$-distance between two functions, we instead opt to compare them via the \textit{relative entropy}. The relative entropy is obtained by seeking an affine functional of $u$, $\ell(u)$, such that the quantity $\mathcal{H}[u]  -  \mathcal{H}[v]  -\ell(u)$ is nonnegative and takes its minimum value of $0$ at $u = v$. This yields the following definition for the relative entropy: 
\begin{align}
\label{re_functional}
 \mathcal{H}[u | v ]  :=  \mathcal{H}[u] -  \mathcal{H}[v] -  \langle  \mathcal{H}^{\prime}[v] , u-v \rangle,
\end{align}
where we mention that the \textit{relative entropy density} $h (\cdot\, |  v)$ is related to the entropy density $h$ as 
\begin{align}
\label{relative_entropy_density}
h (u \, | \, v) = h ( u ) - h(v) - \langle h^{\prime}(v) , u-v \rangle.
\end{align}
The relative entropy is well-suited for obtaining estimates in our context because it satisfies an estimate similar to \eqref{entropy_estimate} --an observation that has already been exploited, e.g., in the uniqueness theory for solutions of cross-diffusion or reaction-diffusion systems \cite{CJ_2017,JZ_2016,JZ_2017, J_2017} or to obtain (exponential) convergence rates to equilibrium \cite{CJM_2001}. 

While cross-diffusion systems with entropy structure have been the topic of much study in recent years --see the contributions already listed above, the survey article \cite{J_overview_2017}, or the book \cite{J_book}-- outside of certain examples or under very restrictive conditions on $A(u)$ in \eqref{cross_diffusion_system}, not much is known about the regularity of weak solutions. As the ultimate goal in much of the work on cross-diffusion systems is the existence of global solutions, the importance of H\"{o}lder regularity results may, e.g., be highlighted by a result of Amann \cite[Thm. 15.3]{A_book}, which links the extendability of a local solution to uniform in time bounds for certain spatial H\"{o}lder norms. In the current contribution our main goal is to show that we can, under natural assumptions --which we expand on below--, adapt classical methods to obtain partial H\"older regularity for bounded weak solutions of \eqref{cross_diffusion_system}, when there is an entropy structure. The naive game-plan is to, within the arguments of Giaquinta and Struwe \cite{GS_82}, replace the use of energy methods by entropy methods --i.e., we replace energy estimates by entropy estimates and the use of the squared $L^2$-distance by that of the relative entropy. Throughout the course of executing this strategy, we find that it is necessary to regularize the entropy structure --we call this regularized entropy the ``glued entropy''. 

The assumptions that we place on \eqref{cross_diffusion_system} are tailored in order to include as many examples of cross-diffusion systems with entropy structure as possible. To give some examples of ``admissible'' cross-diffusion systems, we remark that within our framework we are able to treat both examples of volume-filling systems and also of non volume-filling systems: On the volume-filling side, our methods yield the partial $C^{1,\alpha}$-regularity of weak solutions of the Maxwell-Stefan model (Example \ref{MS_model}) and also of weak solutions of the multi-species diffusion model with size exclusion that was recently studied by Hopf and Burger \cite{Hopf_Burger_2021} (Example \ref{HB_model}). On the non volume-filling side, we obtain partial $C^{1,\alpha}$-regularity for bounded weak solutions of the two-component Shigesada-Kawasaki-Teramoto (SKT) model (Example \ref{SKT}), for bounded weak solutions of the two-component semiconductor model (with electron-hole scattering) derived by Reznik \cite{R_95} (Example \ref{semiconductor}), and, finally, for bounded weak solutions of the  regularized version of the Patlak-Keller-Segel model in two dimensions studied in \cite{HJ_2011} (Example \ref{HJ_model}). We remark that it is shown in \cite{HJ_2011} that the additional cross-diffusion term that is added into the Patlak--Keller-Segel model prevents blow-up in the parabolic-elliptic model. 

\subsection{Overview}
In Section \ref{main_results}, we give the main results of this article: In particular, in Theorems \ref{theorem_MS} - \ref{Theorem_examples_5} we state the application of our partial regularity theory to five examples of cross-diffusion systems (already described above). In Section \ref{strategy_section}, we introduce the regularization of the entropy structure which we require in order to emulate the methods of Giaquinta and Struwe \cite{GS_82} --this is the ``glued entropy''. The definition of the ``glued entropy'' is contained in Section \ref{motivation}, a heuristic construction is described in Section \ref{construction}, and sufficient conditions for making this construction rigorous are given in Section \ref{sufficient_conditions}. In Section \ref{main_partial}, we give the most general version of our partial regularity results --the application of which results in Theorems \ref{theorem_MS} - \ref{Theorem_examples_5}. The argument for this, most general, version of our results is outlined in Section \ref{campanato}. It follows, in particular, through a Campanato iteration that is infused with entropy methods, which are applied w.r.t.\,the ``glued entropy''. In Section \ref{proof_Prop_1}, we prove that, under the conditions given in Section \ref{sufficient_conditions}, the construction of Section \ref{construction} yields a ``glued entropy'', in the sense of Section \ref{motivation}. In Section \ref{proofs_main_sec}, we give the proofs of Theorems  \ref{theorem_MS} - \ref{Theorem_examples_5}. In Sections \ref{poincare} - \ref{theorem_7_proof}, we perform the Campanato iteration that is outlined in Section \ref{campanato} --i.e., these sections contain the argument for the most general version of our regularity result.

\subsection{Notation} 
We will use the notation
\begin{align*}
\Lambda: = \Omega \times (- T,0), \, \, \, \textrm{where }\,  \Omega \subset \R^d \, \textrm{ and } \, T>0.
\end{align*}
Furthermore, a point $z_0 \in \Lambda$ can be decomposed as $z_0 = (x_0, t_0)$ for $x_0 \in \Omega$ and $t_0 \in (-T,0)$. For $R>0$ and a point $z_0 \in \Lambda$, we then let 
\begin{align*}
\Gamma_R(t_0) := ( t_0 - R^2, t_0) \, \, \,  \textrm{and} \, \, \, B_R(x_0) := \big\{x \in \R^d \, \vert \,  |x - x_0| <R \big\}.
\end{align*}
The corresponding parabolic cylinder is
\begin{align}
\Cc_R(z_0):= B_R(x_0) \times \Gamma_R(t_0)
\end{align}
with the parabolic boundary $\partial^P\Cc_R(z_0)$ of $\Cc_R(z_0)$ given by 
\begin{align}
\partial^P\Cc_R(z_0):=( B_R(x_0) \times \{t_0 - R^2\}) \cup ( \partial B_R(x_0) \times [t_0 - R^2, t_0)).
\end{align}
We also use the notation $\R^n_+ := \{ \rho \in \R^n : \rho_i > 0 \text{ for all } i = 1, \ldots, n \}$.

We use two notions and corresponding notations for the average of a function on a parabolic cylinder. The first notion is the standard one and, for a function $u$, radius $R>0$, and point $z_0 \in \Lambda$, is given by 
\begin{align}
(u)_{z_0, R} := \fint_{\Cc_R(z_0)} u \dd z = \frac{1}{|\Cc_R(z_0)|}\int_{\Cc_R(z_0)} u \dd z. 
\end{align}
The second notion is a weighted average: For a point $x_0 \in \Omega$ we introduce a cut-off function $\chi_{x_0} \in C^{\infty}_0(B_2(x_0))$ such that $\chi_{x_0} \equiv 1$ on $B_{1}(x_0)$ and $|\nabla \chi_{x_0}| \leq 2$.  Rescaling, we then let 
\begin{align}
\label{cut_off}
\chi_{x_0,R} := \chi_{x_0}\Big(\frac{\cdot}{R}\Big) \, \, \,  \textrm{for } R>0,
\end{align}
where we notice that $\chi_{x_0,R}$ is supported in $B_{2R}(x_0)$. This allows us to define the time-dependent weighted average on balls
\begin{align}
\label{weighted_averages}
(\tilde{u})_{x_0,R} (t) :=\frac{ \int_{B_{2R}(x_0)} u(x,t) \chi^2_{x_0,R} \dd x }{ \int_{B_{2R}(x_0)}  \chi^2_{x_0,R} \dd x }. 
\end{align}
This notation and the use of the weighted average is taken from \cite{GS_82}.

We use the notation $`` f \lesssim g "$ to denote $``f \leq C(d, n, A, \epsilon) g"$. Here $\epsilon>0$ is determined by the availability of a glued entropy density $h_{\epsilon}$, which we introduce in Section \ref{strategy_section}.

Throughout this paper, we use \textit{parabolic} H\"{o}lder spaces; i.e., they are defined in terms of the \textit{parabolic metric}
\begin{align}
\label{p_metric}
\delta(z_0, z_1) = \max\big\{|x_0 - x_1|, |t_0 - t_1|^{\frac{1}{2}} \big\}. 
\end{align}

We let $H^{k}(\cdot)$ for $k\geq0$ denote the $k$-dimensional Hausdorff measure defined in terms of $\delta$ given in \eqref{p_metric}. In particular, for $\Lambda \subset \R^{d+1}$ we have that 
\begin{align}
\label{hausdorff_measure}
H^{k}(\Lambda) = \lim_{ \epsilon \rightarrow 0 } \inf \big\{ \sum_{i}\delta(\Lambda_i)^{k} \, \vert \,  \Lambda \subset \cup_i \Lambda_i  \textrm{ and } \delta(\Lambda_i) < \epsilon \big\} ,
\end{align}
where $\delta(\Lambda_i)$ denotes the diameter of the set $\Lambda_i$ w.r.t\,the metric $\delta$.

We make use of the following subspaces of $\mathbb{R}^n$:
\begin{align}
\label{defn_subspace}
\Xi_0 := \Big\{ \rho \in \R^n \, :\, \sum_{i=1}^n \rho_i = 0\Big\} \quad \text{and} \quad \Xi_1 := \Big\{ \rho \in \R^n \, :\, \sum_{i=1}^n \rho_i = 1\Big\}.
\end{align}

\section{Main results}
\label{main_results} 

We first give the definition of a weak solution of \eqref{cross_diffusion_system}. Since we only work with nonnegative solutions, we implicitly include this property in the definition. 

The notion of weak solution of \eqref{cross_diffusion_system} that we use is as follows:
\begin{definition}{Weak solution}
\label{weak_solution} 
A {\em weak solution} to \eqref{cross_diffusion_system}  is a 
function $u: \Omega\times [0, T) \to\overline{\mathcal{D}}$ with
$u\in L^2(0,T; H^1(\Omega; \R^n))\cap L^{\infty}(0,T; L^2(\Omega; \R^n))$ and  $\pa_t u\in L^2(0,T; H^1(\Omega;\R^n)')$ such that 
\begin{align}
\label{weak_formulation}
\int_0^T\langle\pa_t u,\phi\rangle\, dt + 
\int_0^T\int_\Omega\nabla\phi : A(u)\nabla u\, dx\, dt
= \int_0^T\int_\Omega f(u)\cdot\phi\, dx\,dt,
\end{align}
for any $\phi\in L^2(0,T; H^1(\Omega; \R^n))$.
Here, $\langle\cdot,\cdot \rangle$ denotes the dual pairing of $H^1(\Omega)'$ and $H^1(\Omega)$.
\end{definition}
\noindent We point out that $u\in C^{0}([0,T]; L^2(\Omega; \R^n))$ thanks to \cite[Prop. 23.23]{Z_90} and, therefore, the strong $L^2(\Omega)$-limit $\lim_{t\to 0}u(\cdot,t)$ exists.\\

We now list five examples of systems to which our partial regularity theory may be applied and give the corresponding results. For practitioners interested in a cross-diffusion system not listed in this section, we refer them to Theorem \ref{Theorem_1} --this contains the most general description of the cross-diffusion systems to which our methods apply.

\subsection{Results for volume-filling systems}
\label{vol_fill_examples}

As already mentioned above, we apply our methods to two examples of volume-filling systems --the first, the Maxwell-Stefan model, we call ``implicit'', since, to write the system in the form \eqref{cross_diffusion_system}, we must invert a flux-gradient relation. Our second example, which is taken from a recent paper by Hopf and Burger \cite{Hopf_Burger_2021}, we call ``explicit'', since it is already written in the form \eqref{cross_diffusion_system}.\\

For volume-filling systems, we will always use $\dom = (0,1)^n$.

\subsubsection{Maxwell-Stefan system} 

We consider the Maxwell-Stefan model, one of the most well-studied cross-diffusion systems.

\begin{example}[Maxwell-Stefan model] \label{MS_model} The Maxwell-Stefan model describes the diffusive evolution of a multicomponent mixture --e.g., a gaseous or fluid mixture \cite{S_1871,M_1866}. The model is given by 
\begin{align}
\label{Maxwell_Stefan}
\partial_t u_i + \nabla \cdot J_i = f_i(u)\quad  \textrm{in} \quad \Omega \times (0,T), 
\end{align}
for $i = 1, \ldots, n$ and $T>0$. The components of $u$ represent the molar concentrations of the different species, which are related to the fluxes $J: \Omega \times (0,T) \rightarrow \R^{n \times d}$ by the relation 
\begin{align}
\label{flux_gradient_condition}
 \nabla u_i = -\sum_{j=1}^n \frac{u_j J_i - u_i J_j}{D_{ij}} = -\sum_{j=1}^n (M_{\rm{MS}}(u) )_{\textit{ij}}J_j,
\end{align}
where, for $i \neq j$, the interspecies diffusion coefficients satisfy $D_{ij} = D_{ji}$ and $D_{ij}>0$. One, furthermore, imposes the condition that 
\begin{align}
\label{volume_filling_constraint}
\sum_{i=1}^n J_i(y) = 0,
\end{align}
for any $y \in \dom$. The model is considered with no-flux boundary conditions, i.e. we have
\begin{align}
\label{boundary_conditions_MS}
 \nu \cdot J_i = 0 \quad  \textrm{on } \partial \Omega \times  (0, T), \quad \text{for } i= 1, \ldots, n,
\end{align}
and with a measurable initial condition $u_0$ satisfying 
\begin{align}
\label{ic_volume_filling}
u_0 \geq 0 \quad \text{and} \quad  \sum_{i=1}^n u_{0,i} = 1. 
\end{align}
\end{example}
Notice that under the additional assumption that the reproduction rates $f_i \in C^0(\dom; \R)$ satisfy 
\begin{align}
\label{reproduction_rates}
\sum_{i=1}^n f_i(y) = 0, \quad \text{for any} \, \, \,  y \in \dom,
\end{align}
the constraint \eqref{volume_filling_constraint}, in conjunction with \eqref{Maxwell_Stefan}, yields that 
\begin{align}
\label{volume_perserving}
\sum_{i=1}^n u_i(\, \cdot \,, t) = 1, \quad \text{for any } t\in (0,T).
\end{align}
When the components of $u$ are nonnegative, the condition \eqref{volume_perserving} clearly implies that each component is bounded. We remark that it is shown in \cite[Section 6]{B_2010} that if the $f_i$ satisfy a quasi-positivity condition ($f_i(u) \geq 0$ whenever $u_i =0 $ and $u_j \geq 0$ for $j\neq i$), then classical solutions of \eqref{Maxwell_Stefan} remain nonnegative as long as they exist. The nonnegativity of the weak global solutions of \eqref{Maxwell_Stefan} that are constructed via the boundedness-by-entropy method is a natural consequence of the method \cite{JS_2012}.

The advantage of using the Maxwell-Stefan approach to model the diffusive dynamics of a multispecies mixture, as opposed to the use of a Fickian model, is that it readily captures ``uphill diffusion'' --observed experimentally, e.g., in the 60s by Duncan and Toor (see \cite{DT_1962} or, for a description of the experiment, \cite[Section 2]{BGS_2012}). Mathematically, the reason that the Maxwell-Stefan model captures this behavior is due to the nonlinearity of the flux-gradient relation \eqref{flux_gradient_condition}. For a derivation of the Maxwell-Stefan system using interspecies force balances see \cite{B_2010} or \cite[Appendix A]{JS_2012}; for a derivation of the Maxwell-Stefan system via kinetic theory see \cite{BGS_2015}. 

The Maxwell-Stefan system has an entropy structure in the following sense: Using the standard Boltzmann entropy, i.e. letting
\begin{align}
\label{B_entropy}
h(u) =  \sum_{i=1}^n h_i(u_i)=u_1 ( \log(u_1) -1) + \ldots + u_n ( \log(u_n) -1),
\end{align}
we see that there exist $\lambda>0$ and $\delta\geq 0$ such that, for any $y \in \dom \cap \Xi_1$ --$\Xi_1$ defined in \eqref{defn_subspace}-- and $\rho \in \R^n$, it holds that 
\begin{align}
\label{hypocoercive_MS}
\rho \cdot h^{\prime \prime}(y) M_{\rm{MS}}(y) \rho \geq \lambda |\rho|^2 - \delta\Big( \sum_{i=1}^n \rho_i \Big)^2.
\end{align}
The relation \eqref{hypocoercive_MS} is, in fact, quite easy to see --going by \eqref{flux_gradient_condition}, for $i,j = 1, \ldots, n$, we have that 
\begin{align}
\label{M_MS}
(M_{\rm{MS}}(y))_{ij} =
\begin{cases}
\sum_{k= 1, k \neq i }^n  \frac{y_k}{D_{ik}} & \text{if } i = j,\\
- \frac{y_i}{D_{ij}} & \text{if } i \neq j.
\end{cases}
\end{align} 
Letting $D:= \max_{i,j = 1, \ldots, n, i \neq j} D_{ij}$, we may then write 
\begin{align}
\label{calc_MS}
\begin{split}
\rho \cdot h^{\prime \prime}(y) M_{\rm{MS}}(y) \rho &= \sum_{i,j = 1}^n \rho_i h_i^{\prime \prime}(y_i) (M_{\rm{MS}}(y))_{ij} \rho_j\\
 & = \frac{1}{2} \sum_{i,j =1}^n \frac{y_i y_j}{D_{ij}} \Big( \frac{\rho_i}{y_i} - \frac{\rho_j}{y_j}\Big)^2\\
& \geq \frac{1}{2D} \sum_{i,j = 1}^n y_i y_j \Big( \frac{\rho_i}{y_i} - \frac{\rho_j}{y_j}\Big)^2 \geq \frac{1}{D} \Big[  \sum_{i = 1}^n \rho_i^2 -  \Big( \sum_{i=1}^n\rho_i \Big)^2 \Big].
\end{split}
\end{align}
Notice that \eqref{hypocoercive_MS}, unlike the entropy condition \eqref{main_entropy_condition_non}, is only a hypocoercivity condition. 

While use of the Maxwell-Stefan system has permeated its way into many applied fields (including pulmonology \cite{C_1980, BGG_2010}), due to the many mathematical challenges inherent to the model, rigorous treatments are quite recent. Making some assumptions on the structure of the nonlinearities and that the initial data is close to equilibrium, Giovangigli proved the existence of unique global solutions to the whole-space Maxwell-Stefan system \cite[Theorem 9.4.1]{G_1999}. For general initial data, the existence of unique local  solutions is shown in \cite{B_2010}. In \cite{B_2010}, in order to use the classical local well-posedness theory of Amann \cite{A_1989,A_1990}, it is necessary to establish the invertibility of the flux-gradient condition \eqref{flux_gradient_condition} --This is done by using Perron-Frobenius theory to analyze the spectrum of $M_{\rm{MS}}(u)$ for positive $u$ satisfying \eqref{volume_perserving}. For $n=3$ and $D_{12} = D_{13}$, the existence and long-time behavior of unique global solutions is addressed in \cite{BGS_2012}. In this case, the 3-component system reduces to two equations: a heat equation and a drift-diffusion type equation. Combining the techniques of \cite{B_2010} with the machinery of the boundedness-by-entropy method, J\"ungel and Stelzer proved the existence of global weak solutions for general initial data and $n>1$; they also proved exponential convergence rates to equilibrium \cite{JS_2012}. We remark that the existence and long-time behavior results in \cite{JS_2012} are contingent on the additional condition that $\sum_{i=1}^n f_i(u) \log(u_i) \leq  0$ (with respect to the existence result this corresponds to the condition $(H3)$ in \cite{J_2014}, which, in turn, is related to a quasi-positivity condition on the $f_i$).  Following \cite{JS_2012}, in \cite{CEM_2020} a finite volume scheme is given that preserves many of the properties of the continuum Maxwell-Stefan model (including a discrete entropy-entropy dissipation inequality).\\

Our main result for the Maxwell-Stefan system is:

\begin{theorem}
\label{theorem_MS} 
Let $\alpha \in (0,1)$ and $u$ be a weak solution of the Maxwell-Stefan system; i.e. $u$ solves \eqref{Maxwell_Stefan} with the flux-gradient relation \eqref{flux_gradient_condition} and the volume-filling constraint \eqref{volume_filling_constraint}. Furthermore, let the initial data satisfy \eqref{ic_volume_filling} and $f \in C^0(\left[0,1\right]^n; \R^n)$ satisfy \eqref{reproduction_rates}. Then, there exists $\Lambda_0 \subseteq \Omega \times (0,T)$ with full $d+2$-dimensional parabolic Hausdorff measure (see \eqref{hausdorff_measure}) such that $u \in C^{1,\alpha}(\Lambda_0)$. In fact, the singular set, $\Omega \times (0,T) \setminus \Lambda_0$, is even of parabolic Hausdorff dimension less than $d - \gamma$ for some $\gamma>0$.
\end{theorem}

\noindent This theorem may, e.g., by applied to the global weak solutions provided by J\"{u}ngel and Stelzer \cite{JS_2012}.

\subsubsection{Multi-species diffusion model with size exclusion} We consider the cross-diffusion model recently studied by Hopf and Burger \cite{Hopf_Burger_2021}.

\begin{example}[Multi-species diffusion with size exclusion] 
\label{HB_model}
This model is given by \eqref{cross_diffusion_system} with $f=0$ and the diffusion matrix 
\begin{align}
\label{A_HB}
(A_{\rm{HB}}(u))_{ij} =
\begin{cases}
\sum_{k=1,k\neq i}^n K_{ik}u_k &\text{if } i = j,\\
- K_{ij} u_i & \text{if } i \neq j,
\end{cases}
\end{align}
which --up to the replacement $K_{ij} = D_{ij}^{-1}$-- is the same as \eqref{M_MS}. It is assumed that $K_{ij}  \geq 0$, for $i,j =1, \ldots, n$, and that $K_{ij} =K_{ji}$. The model is considered with the no-flux boundary data 
\begin{align}
\label{no_flux_bd}
\nu \cdot A_{\rm{HB}}(u) \nabla u = 0 \, \, \, \textrm{on} \, \, \, \partial \Omega \times  (0, T),
\end{align}
with a measurable initial condition $u_0$ satisfying \eqref{ic_volume_filling}.
\end{example}

Notice that by adding the equations for the $u_i$ with \eqref{A_HB} and $f\equiv0$, the symmetry of the $K_{ij}$ implies that the model is mass-preserving (in the sense of \eqref{volume_perserving}).

Establishing the presence of an entropy structure similar to \eqref{hypocoercive_MS} requires further discussion concerning the non-degeneracy of the $K_{ij}$: In \cite{Hopf_Burger_2021} the existence of global weak solutions and their long-time behaviour is addressed under a rather weak assumption, called (H3), which postulates that at least one species interacts with all the others. The loosening of the standard non-degeneracy condition (``full interaction'': $K_{ij} >0$ for $i,j = 1, \ldots, n$) is a notable feature of their setting as the entropy condition \eqref{hypocoercive_MS} breaks down and the situation is much more delicate. Since in our work it is essential to have access to a condition of the form \eqref{hypocoercive_MS}, we cannot work with only the assumption (H3), and must be in the regime of ``full interaction''. In this case, the similarity between $M_{\rm{MS}}$ and $A_{\rm{HB}}$, shows there is an entropy structure in the sense of \eqref{hypocoercive_MS} with the entropy density \eqref{B_entropy}.  

 For a thorough literature review on systems of the above type we refer the reader to \cite[Section 1.2]{Hopf_Burger_2021}. Here, we mention only that local strong solutions were already provided by Amann \cite{A_1990} and that, under the assumption of ``full interaction'', the existence of global weak solutions has already been settled in \cite{J_2014}. Furthermore, in \cite{Hopf_Burger_2021} a weak-strong stability estimate is provided in the setting of ``full interaction''. Attempting to extend our partial regularity result to the admittedly more interesting setting of the assumption (H3) of Hopf and Burger, might be a interesting question for future investigation.\\

Our main result for the multi-species diffusion model with size exclusion from Example \ref{HB_model} is:

\begin{theorem}
\label{Theorem_examples_2} 
Let $\alpha \in (0,1)$ and $u$ be a weak solution of \eqref{cross_diffusion_system} with diffusion coefficients given by \eqref{A_HB} with $K_{ij} >0$, for $i,j = 1, \ldots, n$, and $f =  0$. Furthermore, assume that the initial data satisfies \eqref{ic_volume_filling}. Then, there exists $\Lambda_0 \subseteq \Omega \times (0,T)$ with full $d+2$-dimensional parabolic Hausdorff measure (see \eqref{hausdorff_measure}) such that $u \in C^{1,\alpha}(\Lambda_0)$.  In fact, the singular set, $\Omega \times (0,T) \setminus \Lambda_0$, is even of parabolic Hausdorff dimension less than $d - \gamma$ for some $\gamma>0$.
\end{theorem}

\noindent We remark that it would be possible to replace the condition that $f = 0$ by $f$ satisfying \eqref{reproduction_rates} and $f \in C^0(\left[0,1\right]^n; \R^n)$, as in Theorem \ref{theorem_MS}.

\subsection{Results for non volume-filling systems}

We apply our methods to three examples of non volume-fillings systems: the 2-component SKT model for population dynamics, a semiconductor model with electron-hole scattering, and the Patlak-Keller-Segel model with an additional cross-diffusion term (in $d=2$). In general, the disadvantage in the application of our theory to non volume-filling systems is that the solutions are not naturally bounded; the partial regularity theory that we develop is only applicable to bounded weak solutions. \\

For non volume-filling systems the natural choice for $\dom$ is $\mathbb{R}^n_+$; since we only consider bounded solutions, for these we use $\dom = (0,d_1) \times \ldots \times (0,d_n)$ for some $d_1, \ldots, d_n >0$ (such that $\rm{Range}(\it{u}) \subseteq \overline{\dom}$).

\subsubsection{Shigesada-Kawasaki-Teramoto model}
 
We consider the 2-component SKT model for population dynamics --one of the prototypical examples of a cross-diffusion system. 
 
\begin{example}[Shigesada-Kawasaki-Teramoto model] \label{SKT}

This model is used to describe the evolution of interacting subpopulations --Here, we give it for $n=2$.  The model is given by \eqref{cross_diffusion_system} with the diffusion matrix
\begin{align}
\label{A_SKT}
   A_{\rm{SKT}}(u) = 
  \left[ {\begin{array}{cc}
   \alpha_{10} + 2\alpha_{11} u_1  + \alpha_{12} u_2&  \alpha_{12} u_1 \\
 \alpha_{21} u_2 &   \alpha_{20} + \alpha_{21} u_1  + 2\alpha_{22} u_2
  \end{array} } \right],
\end{align}
where we assume that each $\alpha_{ij}>0$. One usually considers this model with Lotka-Volterra source terms
\begin{align*}
f_i(u) = (\beta_{i0} - \beta_{i1} u_1 - \beta_{i2} u_2 )u_i \qquad \textrm{for} \quad   i = 1,2,
\end{align*}
where the $\beta_{ij} \geq 0$. The model is considered with no-flux boundary conditions (see \eqref{no_flux_bd}) and a nonnegative measurable initial condition. 
\end{example}

A calculation shows that the $2$-component SKT model has an entropy structure with 
\begin{align}
\label{entropy_SKT}
h(u) = \sum_{i=1}^2 h_i(u_i) = \frac{u_1}{ \alpha_{12} } (\log(u_1) - 1) + \frac{u_2}{\alpha_{21}} (\log(u_2) - 1)
\end{align}
and that there exists $\lambda>0$ such that 
\begin{align}
\label{entropy_condition_SKT}
\rho \cdot h^{\prime \prime}(y) A_{\rm{SKT}}(y) \rho \geq \lambda |\rho|^2,
\end{align}
for any $y  \in \mathbb{R}^2_+ $ and $\rho \in \R^2$. 


Local solutions of the SKT system were constructed by Amann \cite{A_1990}. Global solutions have only been constructed under rather restrictive assumptions on the form of $A_{\rm{SKT}}$ and the spatial dimension $d$. For an extensive literature review we point an interested reader to \cite{Yamada_2009}. One situation that has been well-studied is when $A_{\rm{SKT}}$ is upper-triangular --in particular, when $\alpha_{11} > 0$ and $\alpha_{21} =0$. This case has been treated, e.g., in \cite{LNW_1998, CLY_2003, CLY_2004, LNN_2003, P_2008} under additional restrictions on the spatial dimension. Restrictions on the spatial dimension were dropped by Hoang, Nguyen, and Phan in \cite{HNP_2015}.  In terms of weak solutions, in the case that $d=1$, these were constructed in  \cite{GGJ_2003}. This result was extended to arbitrary spatial dimension in \cite{CJ_2004,CJ_2006} --the second of these works considered the case without self-diffusion, i.e. assuming that $\alpha_{11} = \alpha_{22}  = 0$.  The SKT model also falls within those models treatable via the boundedness-by-entropy method (see \cite[Section 2.2]{J_2014} or \cite[Section 4.5]{J_book}). In \cite{JZ_2016} the third author and J\"{u}ngel derived sufficient conditions on the $\alpha_{ij}$ in order to obtain bounded weak solutions.

With respect to regularity results for the SKT model, we mention the work of Le and Le and Nguyen on the H\"older regularity of solutions to cross-diffusion systems. We first highlight \cite{D_2005} in which Le uses the classical methods of Giaquinta and Struwe \cite{GS_82} to obtain a partial regularity result for bounded weak solutions of the SKT system under the assumption of dominating self-diffusion. From our point of view, the reason that the results in \cite{D_2005} require dominating self-diffusion is that, to the best of our knowledge, the methods do not exploit the entropy structure of the model and rely rather on classical energy methods. Later, under similar assumptions on the $\alpha_{ij}$, the result in \cite{D_2005} was upgraded to everywhere H\"older regularity of bounded solutions in \cite{DN_2006}. Compared to \cite{D_2005}, in the current work we do not require any assumptions on the $\alpha_{ij}$ other than strict positivity.\\

Our main result for the SKT model is:

\begin{theorem}
\label{Theorem_examples_3} 
Let $\alpha \in (0,1)$. Furthermore, let $A_{\rm{SKT}}$ be given by  \eqref{A_SKT}, where we assume that $\alpha_{ij}>0$ for $i = 0,1,2$ and $j = 1,2$, and $f \in C^0(\left[0,1\right]^n; \R^n)$. For $u$, a bounded weak solution of the SKT model given by \eqref{cross_diffusion_system}, there then exists $\Lambda_0 \subseteq \Omega \times (0,T)$ with full $d+2$-dimensional parabolic Hausdorff measure (see \eqref{hausdorff_measure}) such that $u \in C^{1,\alpha}(\Lambda_0)$. In fact, the singular set, $\Omega \times (0,T) \setminus \Lambda_0$, is even of parabolic Hausdorff dimension less than $d - \gamma$ for some $\gamma>0$.
\end{theorem}

\noindent Theorem \ref{Theorem_examples_3} yields, e.g., the partial $C^{1,\alpha}$-regularity of the solutions in \cite{JZ_2016}.

\subsubsection{Semiconductor model with electron-hole scattering}

We consider the semiconductor model derived by Reznik \cite{R_95}.

\begin{example}[Semiconductor model with electron-hole scattering] 
\label{semiconductor}  This is a model for the current flow through a semiconductor under the influence of strong electron-hole scattering (EHS). It is given by \eqref{cross_diffusion_system} with $n=2$ and 
\begin{align}
\label{A_semi}
   A_{\rm{SC}}(u) = \frac{1}{1 + \mu_2 u_1 + \mu_1 u_2} 
  \left[ {\begin{array}{cc}
  \mu_1 (1 - \mu_2 u_1)& \mu_1 \mu_2 u_1 \\
\mu_1 \mu_2 u_2 &   \mu_2(1 +\mu_1u_2)
 \end{array} } \right],
\end{align}
where $u_1$ and $u_2$ represent the electron and hole densities and $\mu_1, \mu_2>0$ are the mobility constants.  The model is considered with no-flux boundary conditions (see \eqref{no_flux_bd}) and a nonnegative measurable initial condition. 
\end{example}

For a physical discussion of the effect of strong EHS on the carrier transport in a semiconductor we refer the interested reader to, e.g., \cite{KS_1992, MRP_1987}. We remark that the above model has a kinetic derivation, starting from the semiconductor Boltzmann equation with a collision operator that incorporates strong EHS. The model was originally derived by Reznik in \cite{R_95}. For more details we refer the interested reader to \cite[Section 2.2]{J_2014}.

It can easily be checked that, for $h$ given by \eqref{B_entropy} (with $n=2$) and $\rho \in \R^2$, we have that 
\begin{align}
\label{entropy_structure_SC}
\rho \cdot h^{\prime \prime}(u) A_{\rm{SC}}(u) \rho = \frac{1}{1 + \mu_2 u_1 + \mu u_2} \Big( \frac{\mu_1}{u_1} \rho_1^2 + \frac{\mu_2}{u_2} \rho_2^2 + \mu_1 \mu_2 (\rho_1 + \rho_2)^2\Big)
\end{align}
Using this entropy structure, the existence of global nonnegative weak solutions has been shown in \cite{CJ_2007}.\\

Our main result for the semiconductor model is:

\begin{theorem}
\label{Theorem_examples_4} 
Let $\alpha \in (0,1)$ and $u$ be a bounded weak solution of \eqref{cross_diffusion_system} with diffusion coefficients given by \eqref{A_semi} and $f \in C^0(\left[0,1\right]^n; \R^n)$. Then, there exists $\Lambda_0 \subseteq \Omega \times (0,T)$ with full $d+2$-dimensional parabolic Hausdorff measure (see \eqref{hausdorff_measure}) such that $u \in C^{1,\alpha}(\Lambda_0)$.  In fact, the singular set, $\Omega \times (0,T) \setminus \Lambda_0$, is even of parabolic Hausdorff dimension less than $d - \gamma$ for some $\gamma>0$.
\end{theorem}

\subsubsection{Regularized Patlak-Keller-Segel model}

We consider the regularized version of the Patlak-Keller-Segel model that is studied in \cite{HJ_2011}.

\begin{example}[Patlak-Keller-Segel model with additional cross-diffusion term ($d=2$)] 
\label{HJ_model} The Patlak-Keller-Segel model is used to describe chemotaxis --the migration of cells due to chemical gradients \cite{P_53, KS_70}. In \cite{HJ_2011}, Hittmeier and J\"ungel add an additional cross-diffusion term, thereby changing the entropy structure of the system in a substantial way. The modified model considered in \cite{HJ_2011} is given  by 
\begin{align}
\label{PKS}
\begin{split}
\partial_t u_1 - (\nabla \cdot( A_{\rm{HJ}}(u))_{11} \nabla u_1 + \nabla \cdot (A_{\rm{HJ}}(u))_{12} \nabla u_2)&  = f_1(u)\\
\beta \partial_t u_2 - (\nabla \cdot (A_{\rm{HJ}}(u))_{21} \nabla u_1 + \nabla \cdot (A_{\rm{HJ}}(u))_{22} \nabla u_2)&  = f_2(u),
\end{split}
\end{align}
with 
\begin{align}
\label{A_HJ}
   A_{\rm{HJ}}(u) =
  \left[ {\begin{array}{cc}
  1& -u_1 \\
  \delta & 1   
 \end{array} } \right] \quad \text{and} \quad f(u) = \left( \begin{array}{c} 0 \\ \mu u_1 - u_2 \end{array} \right).
\end{align}
Here, $u_1$ represents the cell density and $u_2$ is the chemical signal concentration, $\beta \geq 0$ is a time parameter, $\mu>0$ is the production rate (of the chemical by the cells), and $\delta>0$. The model is considered for $\Omega \subset \R^2$.  The additional cross-diffusion term in \eqref{PKS} is $``\delta \Delta u_1"$ in the equation for $u_2$. When $\beta = 0$, \eqref{PKS} is called the ``parabolic-elliptic'' model. 

The model is considered with the no-flux boundary data 
\begin{align*}
\nu \cdot \nabla u = 0 \quad \text{on } \partial \Omega \times (0T),  
\end{align*}
with a nonnegative measurable initial condition. 
\end{example}

It is well-known (see, e.g., the review article \cite{H_2003}) that in the classical Keller-Segel model, i.e. when $\delta =0$ in \eqref{PKS}, if the initial cell density is large enough, then, for $d=2,3$, there is finite-time blow-up of the cell-density. For the case that $\Omega \subset \mathbb{R}^2$, it has been shown that there exist global solutions when $\int_{\Omega} u_1 < 4 \pi$ and that there is finite-time blow-up when  $\int_{\Omega} u_1 > 4 \pi$ \cite{NSY_1997}. The parabolic-elliptic model on the whole-space, $\mathbb{R}^2$, has been treated in \cite{BDP_2006,BCM_2008}. The situation that $\Omega \subset \mathbb{R}^3$ has been addressed, e.g., in \cite{CP_2006}, where the $L^{d/2}$-norm of the initial cell density plays the role of the mass from the $2$-dimensional case. 

In \cite{HJ_2011} the cross-diffusion term $``\delta \Delta u_1"$ is inserted into the equation for $u_2$ in order to prevent blow-up of the cell density. Prior to this, various other methods were introduced to prevent blow-up: changing the chemotactic sensitivity, modifying the form of the diffusive term in the equation for $u_1$, and having a model with non-vanishing birth-death --to avoid repetition, we refer the interested reader to \cite{HJ_2011} for a literature review of these techniques. In \cite{HJ_2011}, under some integrability assumptions on $u_0$, global weak solutions of \eqref{PKS} are constructed; in the parabolic-elliptic case these are bounded. In \cite{JLW_2019} the vanishing cross-diffusion limit of \eqref{PKS} is considered.

The methods in \cite{HJ_2011} rely on an entropy structure; i.e. that, for any $y\in \mathbb{R}^2_+$ and $\rho \in \R^2$, we have that 
\begin{align}
\label{entropy_cond_HJ}
\rho \cdot h^{\prime \prime}(y)  A_{\rm{HJ}}(y) \rho \geq \frac{ |\rho|^2}{\max(y_1, \delta)},
\end{align}
for the entropy density 
\begin{align}
\label{entropy_HJ}
h_{\rm{HJ}}(u) = u_1(\log(u_1) -1) +\frac{\beta}{2\delta} u_2^2.
\end{align}

Our main result for the Patlak-Keller-Segel model with additional cross-diffusion, \eqref{PKS}, is:

\begin{theorem}
\label{Theorem_examples_5} 
Let $\alpha \in (0,1)$ and $u$ be a bounded weak solution of \eqref{PKS}. Then, there exists $\Lambda_0 \subseteq \Omega \times (0,T) \subset \mathbb{R}^2 \times (0,T)$ with full $4$-dimensional parabolic Hausdorff measure (see \eqref{hausdorff_measure}) such that $u \in C^{1,\alpha}(\Lambda_0)$.  In fact, the singular set, $\Omega \times (0,T) \setminus \Lambda_0$, is even of parabolic Hausdorff dimension less than $2 - \gamma$ for some $\gamma>0$.
\end{theorem}

\noindent The main application of this theorem is to the solutions constructed in \cite{HJ_2011}. 



\section{Our strategy}
\label{strategy_section}

To obtain the results listed in Section \ref{main_results}, we inject the classical framework of Giaquinta and Struwe with entropy methods --in order to be successful, we find that we must ``regularize'' the entropy structure of the system. In particular, we replace the standard entropy density with a ``glued entropy density'', the point being that the glued entropy density still satisfies an entropy condition similar to \eqref{main_entropy_condition_non}, but also has a bounded Hessian. 

Given a cross-diffusion system of the form \eqref{cross_diffusion_system} with an entropy structure and a bounded weak solution, obtaining a partial regularity result via our methods has two steps: (1) Show that the entropy structure can be regularized (i.e., there exists a ``glued entropy''). And, (2) Apply a result of the following type:\\

\noindent \textbf{Heuristic result:} \textit{Let $u: \Omega \times (0,T) \rightarrow \overline{\dom}$ be a bounded weak solution of the cross-diffusion system \eqref{cross_diffusion_system}. Assuming that there exists a glued entropy, that $A$ is uniformly continuous, and that $f \in C^0(\overline{\dom} ; \R^n)$, the solution $u$ has partial $C^{0,\alpha}$-regularity for any $\alpha \in (0,1)$.}\\

\noindent This heuristic statement is formalized in Theorems \ref{Theorem_1} --under the assumption of the $C^{0,\sigma}$-regularity of $A$, we also obtain the partial $C^{1, \sigma}$-regularity of $u$ (see Theorem \ref{higher_reg}). 

In this section we focus on ``explicit'' volume-filling and non volume-filling systems, which are already in the form \eqref{cross_diffusion_system}. In order to handle ``implicit'' volume-filling systems with our methods, as we will see in Section \ref{implicit_systems}, we must assume that the flux-gradient relation is invertible (see \hyperlink{H0}{(\textbf{H0})} below), as is the case for the Maxwell-Stefan model (Example \ref{MS_model}).\\

\subsection{Derivation of the definition of the glued entropy}
\label{motivation}

Recall that the arguments of Giaquinta and Struwe in \cite{GS_82} rely on energy methods and that they require access to a Caccioppoli-type estimate for solutions of the nonlinear system, as well as for solutions of a corresponding ``frozen'' system. Here, we give the heuristics for obtaining a Caccioppoli-type estimate for a bounded weak solution $u$ of \eqref{cross_diffusion_system} assuming that there exists an entropy density such that \eqref{main_entropy_condition_non} holds (see, e.g., the non volume-filling Examples \ref{SKT}, \ref{semiconductor}, and \ref{HJ_model}) for $y \in \overline{\dom}$ and that $A \in L^{\infty}_{\text{loc}}(\mathbb{R}^n)$. 

\paragraph{Heuristic argument for Caccioppoli-type estimate satisfied by $u$:} Let $z_0 \in \Lambda$ and $R>0$ such that $\Cc_{2R}(z_0) \subset \Lambda$ and, for simplicity, assume that $f = 0$. The idea is to mimic the entropy estimate  
\eqref{entropy_estimate} applied to the relative entropy $\int_{\Omega} h(u \, | (\tilde{u})_{x_0, R}) \dd x$, but within the framework of the argument one usually uses to prove the standard Caccioppoli estimate for nonlinear parabolic systems (see, e.g., \cite[Lemma 2.1]{GS_82}). Notice that the weighted average $(\tilde{u})_{x_0, R}$ has been defined in \eqref{weighted_averages}. 

To mimic \eqref{entropy_estimate}, letting $\eta$ be a specific cut-off function for $\Cc_R(z_0)$ in $\Cc_{2R}(z_0)$, we take the time derivative of $\int_{\Omega} h(u \, |(\tilde{u})_{x_0, R})\eta^2 \dd x$. After some manipulations that are contained in the proof of Lemma \ref{nonlinear_cacc}, this yields that 
\begin{align}
\label{intermediate_cacc}
\begin{split}
 \int_{\Cc_{2R}(z_0)} \eta^2 |\nabla u|^2 \dd z & \lesssim \int_{\Cc_{2R}(z_0)} \eta^2 \nabla u : h^{\prime \prime}(u) A(u) \nabla u \dd z\\
 &\,\, \lesssim \frac{1}{R^2} \int_{\Cc_{2R}(z_0)} \Big(h(u \, |  (\tilde{u})_{x_0, R} ) +\sup_{y \in \overline{\dom}} | h^{\prime \prime}(y) |^2 |u-  (\tilde{u})_{x_0, R} |^2  \Big)\dd z.
 \end{split}
\end{align}
Here, we have used the boundedness of $u$ as $A(u) \lesssim 1$. To obtain a Caccioppoli-type estimate from \eqref{intermediate_cacc}, it would be helpful if 
\begin{align}
\label{conditions_for_glued}
h(u|b) \lesssim |u-b|^2 \textrm{ for any } b \in \overline{\dom}\qquad \textrm{and} \qquad \sup_{y \in \overline{\mathcal{D}}}h^{\prime \prime}(y) \lesssim 1.
\end{align}
Notice that when $h \in C^2(\overline{\dom})$, these two conditions are both satisfied via the definition \eqref{re_functional}. Of course, the conditions in \eqref{conditions_for_glued} do not hold for the Boltzmann entropy \eqref{B_entropy} --in particular, $h^{\prime \prime}(y)$ blows-up as $y \rightarrow 0$. 

\smallskip

\begin{remark}
Notice that in \eqref{intermediate_cacc} we have used that \eqref{main_entropy_condition_non} holds for $y \in \overline{\dom}$, whereas in, e.g., Examples \ref{SKT}, \ref{semiconductor}, and \ref{HJ_model} the entropy condition only holds for $y \in \dom$. This property is included in our definition of the glued entropy (below) and is satisfied due to the continuity properties of the construction provided in Section \ref{construction}.
\end{remark}

\smallskip

Motivated by the above discussion we introduce the following definition:

\begin{definition}{Glued entropy density} 
\label{glued_def}
For a cross-diffusion system of the form \eqref{cross_diffusion_system} and $\dom \subseteq \R_+^n$ convex, a nonnegative convex function $h_{\epsilon}: \overline{\dom} \rightarrow \R$ is a \textit{glued entropy density} if:\\

$\bullet$ \textit{(non volume-filling systems)} The following two conditions hold:

\medskip
\begin{itemize}[leftmargin=.7in]
\item[\hypertarget{C1}{\textbf{(C1)}}] 
There exists $\lambda>0$ such that, for any $\rho \in\mathbb R^n$ and $y \in \overline{\mathcal{D}}$, we have that 
 \begin{align*}
	\rho \cdot h_\epsilon''(y)A(y) \rho \geq \lambda |\rho|^2.
	\end{align*}

\item[\hypertarget{C2}{\textbf{(C2)}}]  There exists $\lambda^{\prime}>0$ such that, for any $\rho \in \R^n$ and $y \in \overline{\mathcal{D}}$, it holds that
 \begin{align}
 \label{positive_entries}
	\rho \cdot h_\epsilon''(y) \rho \geq \lambda^{\prime} |\rho|^2;
	\end{align}
furthermore, for $i,j = 1, \ldots, n$, $\lambda^{\prime} \leq |(h_{\epsilon}^{\prime \prime} (y))_{i,j} |\lesssim 1$. 
\end{itemize}

\medskip

$\bullet$  \textit{(``explicit'' volume-filling systems)} The condition \hyperlink{C2}{(\textbf{C2})} holds and 
\medskip
\begin{itemize}[leftmargin=.7in]
\item[\hypertarget{C1p}{\textbf{(C1$^{\prime}$)}}] 
There exists $\lambda>0$ such that, for any $\rho \in \Xi_0$ and $y \in \overline{\dom}\cap \Xi_1$ (see \eqref{defn_subspace}), we have that 
 \begin{align*}
	\rho \cdot h_\epsilon''(y)A(y) \rho \geq \lambda |\rho|^2.
\end{align*}
\end{itemize}

\end{definition}

\noindent We remark that no separate definition is given for the ``implicit'' volume-filling case, since under the assumption \hyperlink{H0}{(\textbf{H0})} in Section \ref{implicit_systems} we may, in fact, rewrite these as ``explicit'' systems.

\medskip

We explain the condition \hyperlink{C1p}{(\textbf{C1}$^{\prime}$)} for the case of a volume-filling system: Notice that, since in this case no entropy condition of the form \eqref{main_entropy_condition_non} holds, \textit{a priori} the computation \eqref{intermediate_cacc} is not applicable to these systems. In particular, in \eqref{intermediate_cacc}, we must replace the use of the coercivity condition \eqref{main_entropy_condition_non} by, e.g., the hypocoercivity condition \eqref{hypocoercive_MS} --This turns out to be easy, after making the observation that in \eqref{intermediate_cacc} we only use the entropy condition with vectors $\rho_i = (\partial_i u_1, \ldots, \partial_i u_n)$, for $i = 1, \ldots, n$, and that, thanks to \eqref{volume_perserving},  these $\rho_i$ are contained in $\Xi_0$. Furthermore, the entropy condition is applied for $y = u \in \overline{\dom} \cap \Xi_1$. For $y \in \overline{\dom} \cap \Xi_1$ and $\rho \in \Xi_0$, the conditions \eqref{main_entropy_condition_non} and \eqref{hypocoercive_MS} coincide.

\bigskip 

We remark that the condition \hyperlink{C2}{(\textbf{C2})} implies that, for $u,v \in L^2(0,T; H^1(\Omega; \overline{\dom}))$, the relation
\begin{align}
\label{compare}
|u - v|^2 \lesssim h_{\epsilon}(u \vert v) \lesssim  |u- v|^2
\end{align}
holds. Here, $h_{\epsilon}(u  \vert v)$ represents the relative entropy density induced by the glued entropy density $h_{\epsilon}$; i.e. 
\begin{align}
\label{relative_entropy_ep}
h_\epsilon(u\vert v ) := h_\epsilon(u) - h_\epsilon( v )
- \langle h_\epsilon'(v) ,  u - v\rangle .
\end{align}

\subsection{Construction of the glued entropy}
\label{construction}

For our construction, we make certain assumptions on the entropy density of the cross-diffusion system:

\begin{itemize}[leftmargin=.5in]
\item[\hypertarget{H1}{\textbf{(H1)}}] The entropy density $h:\mathcal D\to[0,\infty)$ has the form
	\begin{align}
	\label{entropy}
	h(y):=\sum_{i=1}^nh_i(y_i),
	\end{align}
for $y \in \mathcal{D}$ (with $y_i = y \cdot e_i$), and for $h_i \in C^{2}( (0, d_i) )$ nonnegative and convex. (Recall that $\dom =(0,1)^n$ in the volume-filling case and $\dom =(0, d_1) \times \ldots \times (0,d_n)$, for some $d_i >0$, $i = 1, \ldots, n$, in the non volume-filling case.)  

For $i = 1, \ldots, n$, we assume that either: (\textbf{Case 1}) $h_i^{\prime \prime} (y_i) \rightarrow \infty$ monotonically as $y_i \rightarrow 0$ in such a way that there exists $C \in \R$ for which $h_{i}^{\prime \prime} (\epsilon) \leq C h_{i}^{\prime \prime} (2 \epsilon)$ holds for any $\epsilon >0$, or (\textbf{Case 2}) $h_i''$ is bounded on $[0,d_i]$ in the sense of \hyperlink{C2}{(\textbf{C2})}. 
\end{itemize}

Clearly, under the structural assumptions on the entropy in \hyperlink{H1}{(\textbf{H1})} the two conditions in \hyperlink{C2}{(\textbf{C2})} are equivalent. Furthermore, we remark that in all of the examples in Section \ref{main_results}, except Example \ref{HJ_model}, the Boltzmann entropy is used, and, for $i = 1, \ldots, n$, $h_i''$ blows-up as required in Case 1 above. Case 2 is only included to handle the quadratic $h_2$ in \eqref{entropy_HJ}.

For convenience, we introduce the notations:
\begin{align}
\label{case_1_S}
\mathcal{S} = \left\{1, \ldots, n \right\}, \quad  \mathcal{S}_{1}  = \left\{ i \in \mathcal{S} \, : \, \text{Case 1 in \hyperlink{H1}{(\textbf{H1})} holds}  \right\}, \quad \text{and }  \mathcal{S}_{2} = \mathcal{S} \setminus \mathcal{S}_1.
\end{align}

\paragraph{Intuition behind our construction:} For simplicity, we consider the case of the SKT model (Example \ref{SKT}) with entropy density given by \eqref{entropy_condition_SKT}. Notice that if the components of the solution, the $u_i$, were bounded away from $0$, then the boundedness of the Hessian posited in \hyperlink{C2}{(\textbf{C2})} would hold. Therefore, we seek some additional structure to exploit in the case that the $u_i$ are ``small''. The observation that we make, and the backbone of our analysis, is that as $u \rightarrow 0$, the components $(A_{\rm{SKT}}(u))_{ij} \rightarrow \alpha_{i0} \delta_{ij}$ (see \eqref{A_SKT}). Our strategy for constructing the ``glued entropy density'' is then as follows: Whenever the components $u_i$ are ``small enough'', we view the cross-diffusion system \eqref{cross_diffusion_system} as a perturbation of $n$ decoupled heat equations (for $f = 0$) and, since any convex function is an entropy density of the heat equation, replace the entropy structure given by \eqref{entropy_condition_SKT} by a quadratic entropy density in this regime. This is helpful because the Hessian of the quadratic entropy density clearly satisfies \hyperlink{C2}{(\textbf{C2})}. In practice, we will glue the entropy given by \eqref{entropy_condition_SKT} (used when the $u_i$ are ``large enough'') to the quadratic entropy density (used when the $u_i$ are ``small enough'') --\hyperlink{C2}{(\textbf{C2})} is then easily seen to hold, it is harder to check that  \hyperlink{C1}{(\textbf{C1})} or \hyperlink{C1p}{(\textbf{C1}$^{\prime}$)} survives the gluing.\\

Throughout this article we use $\epsilon>0$ to denote the size of the $u_i$ at which we switch from considering \eqref{cross_diffusion_system} as a perturbation of uncoupled heat equations (for $f = 0$) to viewing it as a cross-diffusion system with entropy structure. We call the glued entropy density that is glued at $\epsilon>0$, $h_{\epsilon}$.\\

\paragraph{Construction:}  To construct the ``glued entropy density'', for each $h_i$ with $i=1, \ldots, n$, we consider the two cases included in \hyperlink{H1}{(\textbf{H1})}:\\

\noindent \textbf{Case 1:\,($i \in \mathcal{S}_1$)} In this case, in order to satisfy \hyperlink{C2}{(\textbf{C2})}, we have to modify the entropy density: Let $\epsilon>0$ be arbitrary and take a partition of unity subordinate to the cover of $\R$ given by $A_1 = (-\infty, 2 \epsilon)$ and $A_2 = (\epsilon, + \infty)$ --this partition of unity consists of $\eta^1_{\epsilon}$ and $\eta^2_{\epsilon}$ and we, furthermore, assume that each $|\nabla \eta^k_{\epsilon}| \lesssim 1/ \epsilon$ for $k = 1,2$. 
We define $h_{\epsilon,i}$ as 
\begin{align}
\label{defn_glued_i}
h_{\epsilon,i} (x):= \int_0^x \int_0^z h_i^{\prime \prime} \big( \epsilon \eta^1_{\epsilon}(y) +  y  \eta^2_{\epsilon}(y) \big) \, \dd y \dd z.
\end{align}

\noindent \textbf{Case 2:\,($i \in \mathcal{S}_2$)} In this case, we do not alter $h_i$ and define $h_{\epsilon,i}(y) := h_i(y)$ for any $y \in \overline{\dom}$.\\

\medskip

Having considered these two cases, we then define the \textit{glued entropy density} as
\begin{align}
\label{defn_glued}
h_\epsilon(u) := \sum_{i=1}^n h_{\epsilon,i}(u_i).
\end{align}
Since in the situation of Case 1 above, we have that 
\begin{align}
\label{hessian}
 h_{\epsilon,i}^{\prime \prime}(u_i) =  h_i^{\prime \prime} \big( \epsilon \eta^1_{\epsilon}(u_i) +  u_i  \eta^2_{\epsilon}( u_i) \big),
\end{align}
if $h_i^{\prime \prime}$ is bounded from above and below on $[\epsilon, d_i]$, then $h_{\epsilon}$ as defined above satisfies \hyperlink{C2}{(\textbf{C2})}.\\


\paragraph{Intuition behind Case 1 revisited:} Before moving on, let us revisit the intuition behind the treatment of Case 1 above --again using the SKT model (Example \ref{SKT}). For simplicity, set $\alpha_{21} = \alpha_{12}=1$,  by \eqref{entropy_SKT}  we then have that
$h_i^{\prime \prime}(u_i) = u_i^{-1}$, for $i = 1,2$. To make sure that $h_{\epsilon,i}^{\prime \prime} \lesssim 1$ on $[0,d_i]$ the most naive ansatz for $h_{\epsilon,i}$ would be 
\begin{align}
\label{naive_ansatz}
h_{\epsilon,i}(x) `` = "\int_0^x \int_0^z  h_i^{\prime \prime} \big( \max\{y\,,\epsilon\} \big) \dd y \dd z.
\end{align}
Choosing the quadratic entropy density $\tilde{h}_{\epsilon,i}(u_i) = (2\epsilon)^{-1} u_i^2$ of the heat equation, we notice that \eqref{naive_ansatz} corresponds to gluing $\tilde{h}^{\prime \prime}_{\epsilon,i}$ to $h_i^{\prime \prime}$ and integrating up the result. Going from the naive ansatz \eqref{naive_ansatz} to the actual definition \eqref{defn_glued_i} is a simple matter of replacing $``\max\{\cdot,\epsilon\}"$ by a smooth gluing so that $h_{\epsilon}\in C^2(\overline{\dom})$.

\vspace{.2cm}

\subsection{Sufficient conditions for the existence of a glued entropy}
\label{sufficient_conditions}

We now give sufficient conditions under which, for a cross-diffusion system of the form \eqref{cross_diffusion_system} with entropy density satisfying \hyperlink{H1}{(\textbf{H1})}, there exists $\epsilon>0$ such that $h_{\epsilon}$ defined in \eqref{defn_glued} satisfies \hyperlink{C1}{(\textbf{C1})} or \hyperlink{C1p}{(\textbf{C1$^{\prime}$})} and \hyperlink{C2}{(\textbf{C2})}.\\

We split the sufficient conditions into two cases: non volume-filling systems and ``explicit'' volume-filling systems. The special case of ``implicit'' volume-filling systems is handled in Section \ref{implicit_systems}.

\paragraph{Conditions for non volume-filling systems:}

\begin{itemize}[leftmargin=.5in]
\item[\hypertarget{H2}{\textbf{(H2)}}]
There exists $\lambda>0$ such that, for any $y \in \dom$ and $ \rho \in \R^n$, we have that
	\begin{align}\label{lb.d2hA}
	\rho \cdot h''(y)A(y) \rho \geq \lambda |\rho|^2.
		\end{align}
\item[\hypertarget{H3}{\textbf{(H3)}}] For $i \in \mathcal{S}_1$, there exist functions $a_i \in C^0(\overline{\dom})$ such that
	\begin{equation}\label{mua}
	\mu := \min_{i \in \mathcal{S}_1}\inf_{\overline{\dom}}a_i>0
	\end{equation}
and the relation 
	\begin{align}\label{hp.A}
	\max_{i \in \mathcal{S}_1, j = 1, \ldots ,n} |A_{ij}(y)- a_i(y)\delta_{ij}||h_i''(y_i)| \lesssim 1 
	\end{align}
holds for any $y \in \overline{\dom}$. Furthermore, for $i,j = 1, \dots, n$,  we assume that $A_{ij} \in C^0(\overline{\dom})$.

\end{itemize}

\medskip

\paragraph{Conditions for ``explicit'' volume-filling systems:} 

Letting $\Xi_1 \subset \mathbb{R}^n$ be defined as in \eqref{defn_subspace}, the conditions are given as:

\begin{itemize}[leftmargin=.5in]
\item[\hypertarget{H2p}{\textbf{(H2$^{\prime}$)}}] 
There exist $\lambda>0$ and $\delta \geq0$ such that, for any $y \in \dom \cap \Xi_1$ and $\rho \in \R^n$, it holds that 
\begin{align}
\label{hypocoercive_MS_conditions}
\rho \cdot h^{\prime \prime}(y) A (y) \rho \geq \lambda |\rho|^2 - \delta\Big( \sum_{i=1}^n \rho_i \Big)^2.
\end{align}

\item[\hypertarget{H3p}{\textbf{(H3$^{\prime}$})}] For $i \in \mathcal{S}_1$, there exist functions $a_i \in C^0(\overline{\dom})$ such that
	\begin{equation}\label{mua}
	\mu := \min_{i \in \mathcal{S}_1}\inf_{\overline{\dom}}a_i>0
	\end{equation}
and the relation 
	\begin{align}\label{hp.M}
	\max_{i \in \mathcal{S}_1, j = 1, \ldots ,n} |A_{ij}(y)- a_i(y)\delta_{ij}||h_i''(y_i)| \lesssim 1 
	\end{align}
holds for any $y \in \overline{\dom} \cap \Xi_1$. Furthermore, for $i,j = 1, \dots, n$, we assume that $A_{ij} \in C^0(\overline{\dom})$. 
\end{itemize}

\medskip

To see that the above conditions are indeed sufficient for the existence of a glued entropy, we must check that  \hyperlink{C1}{(\textbf{C1})} or \hyperlink{C1p}{(\textbf{C1$^{\prime}$})} and \hyperlink{C2}{(\textbf{C2})} are satisfied. First, notice that \hyperlink{C2}{(\textbf{C2})} is guaranteed by the definitions \eqref{defn_glued_i} and \eqref{defn_glued} in conjunction with \hyperlink{H1}{(\textbf{H1})}.

It then only remains to check that $\epsilon>0$ can be chosen in such a manner that \hyperlink{C1}{(\textbf{C1})} or \hyperlink{C1p}{(\textbf{C1$^{\prime}$})} is satisfied. In particular, we will show:

\begin{prop}[Verifying \hyperlink{C1}{(\textbf{C1})} or \hyperlink{C1p}{(\textbf{C1}$^{\prime}$)}]
\label{glued_entropy}
\qquad\\
\vspace{-.3cm}
\begin{enumerate}
\item (non volume-filling systems) Under the conditions \hyperlink{H1}{(\textbf{H1})}, \hyperlink{H2}{(\textbf{H2})}, and \hyperlink{H3}{(\textbf{H3})} and using the definitions \eqref{defn_glued_i} and \eqref{defn_glued} of $h_{\epsilon}$, there exist $\epsilon > 0$ and $\lambda>0$ such that 
	\begin{align}
	\label{prop_1_result}
 \rho \cdot h_\epsilon''(y)A(y) \rho \geq \lambda |\rho|^2,
	\end{align}
	for any $\rho \in \R^n$ and $y \in \overline{\dom}$.\\
\item (``explicit'' volume-filling systems) Under the conditions \hyperlink{H1}{(\textbf{H1})}, \hyperlink{H2p}{(\textbf{H2}$^{\prime}$)}, and \hyperlink{H3p}{(\textbf{H3}$^{\prime}$)} and using the definitions \eqref{defn_glued_i} and \eqref{defn_glued} of $h_{\epsilon}$ and \eqref{defn_subspace} for $\Xi_0$ and $\Xi_1 \subset \mathbb{R}^n$, there exist $\epsilon > 0$ and $\lambda>0$ such that \eqref{prop_1_result} holds for $\rho \in \Xi_0$ and $y \in \overline{\dom} \cap \Xi_1$.
\end{enumerate}
\end{prop}
\noindent The proof of Proposition \ref{glued_entropy} is contained in Section \ref{proof_Prop_1}, and takes advantage of the intuition which we have give in Section \ref{construction} (and which is encoded in \hyperlink{H3}{(\textbf{H3})} or \hyperlink{H3p}{(\textbf{H3}$^{\prime}$)}).

\subsection{Main partial regularity result}
\label{main_partial}

For the applicability of the partial regularity theory that we develop we require some additional assumptions on $A$ and $f$ in \eqref{cross_diffusion_system}:
\begin{itemize}[leftmargin=.5in]
 \item[\hypertarget{H4}{\textbf{(H4)}}]  $A \in C^0(\overline{\dom}; \R^{n \times n})$.
 \item[\hypertarget{H5}{\textbf{(H5)}}] $\displaystyle\max_{i=1, \ldots, n}\sup_{y \in \overline{\dom}} |f_i(y)| \lesssim 1$. 
\end{itemize}
Since we only consider bounded solutions, it suffices, e.g., if $f \in C^0(\overline{\dom}; \R^n)$. (Notice that \hyperlink{H4}{(\textbf{H4})} is actually already included in \hyperlink{H3}{(\textbf{H3})}.)

\medskip

Here is the most general form of our partial $C^{0,\alpha}$-regularity result:

\medskip

\begin{theorem}[Partial $C^{0,\alpha}$- regularity]
\label{Theorem_1} Let $u$ be a bounded weak solution of \eqref{cross_diffusion_system}. Assume that the conditions \hyperlink{H1}{(\textbf{H1})}, \hyperlink{H4}{(\textbf{H4})}, and \hyperlink{H5}{(\textbf{H5})} hold and, furthermore, that one of the following is satisfied:\\

$\bullet$ (non volume-filling systems) \hyperlink{H2}{(\textbf{H2})} and  \hyperlink{H3}{(\textbf{H3})} hold.\\

$\bullet$ (``explicit'' volume-filling systems) \hyperlink{H2p}{(\textbf{H2}$^{\prime}$)} and  \hyperlink{H3p}{(\textbf{H3}$^{\prime}$)} hold, and the volume-filling constraint \eqref{volume_perserving}.\\

Then, there exists an open set $\Lambda_0 \subseteq \Lambda$ such that $u \in C^{0,\alpha}_{\rm{loc}}(\Lambda_0)$ for any $ \alpha \in (0,1)$. Furthermore, there exist $\epsilon_0$ and $\epsilon_1>0$ such that 
\begin{align}
\label{set_condition_1}
\Lambda \setminus \Lambda_0 \subseteq  \Big\{ z_0 \in \Lambda \, \vert \, \displaystyle\liminf_{R \rightarrow 0} \fint_{\Cc_R(z_0)} | u - (u)_{z_0,R}  |^2 \, \textrm{d}z >\epsilon_0 \Big\}
\end{align}
and 
\begin{align}
\label{set_condition_2}
\Lambda \setminus \Lambda_0 \subseteq 
 \Big\{ z_0 \in \Lambda \, \vert \, \displaystyle\liminf_{R \rightarrow 0}   R^{-d} \int_{\Cc_R(z_0)} | \nabla u |^2 \, \textrm{d}z >\epsilon_1 \Big\}.
\end{align}
\end{theorem}

Using standard arguments, we obtain the following corollary:

\begin{corollary}
\label{Theorem_1_corollary}
Under the assumptions of Theorem \ref{Theorem_1}, the singular set $\Lambda \setminus \Lambda_0$ satisfies 
\begin{align}
\label{hausdorff_measure_result}
H^{d- \gamma}(\Lambda \setminus \Lambda_0 )  = 0,
\end{align}
for some $\gamma>0$. 
\end{corollary}

For a higher partial regularity result, we then show that under the additional assumption that the $A_{ij}$ are H\"{o}lder continuous with exponent $\sigma \in (0,1)$,  $\nabla u$ (for $u$ satisfying the assumptions of Theorem \ref{Theorem_1}) satisfies a partial H\"{o}lder continuity result also with exponent $\sigma$.

\begin{theorem}[Partial $C^{1, \sigma}$- regularity]
\label{higher_reg}
Let $i, j = 1, \ldots, n$ and $\sigma \in (0,1)$. We adopt the assumptions of Theorem \ref{Theorem_1} and additionally assume that $A_{ij} \in C_{\rm{loc}}^{0, \sigma}( \mathcal{D})$. Under these conditions, we find that $\nabla u \in C^{0,\sigma}_{\rm{loc}}(\Lambda_0)$, where $\Lambda_0$ is determined in Theorem \ref{Theorem_1}. 
\end{theorem}

\noindent With the tools used to prove Theorem \ref{Theorem_1} at our disposal, the proof of Theorem \ref{higher_reg} is quite classical and a similar argument can be found in \cite[Theorem 3.2]{GS_82}.

\subsection{Strategy of the proof of Theorem \ref{Theorem_1}: A Campanato iteration}
\label{campanato}

Under the conditions of Theorem \ref{Theorem_1} we have access to a glued entropy density $h_{\epsilon}$, as defined in Definition \ref{glued_def}. Throughout our arguments we make use of this glued entropy structure without further notice.

The strategy that we pursue for proving Theorem \ref{Theorem_1} is to use a Campanato iteration. In particular, define the \textit{tilt excess} of $u$ as 
\begin{align}
\label{excess}
\phi(z_0; R) := \fint_{\Cc_R(z_0)} | u -  (u)_{z_0, R} |^2\dd z.
\end{align}
We then show the following: Fix $\alpha \in (0,1)$. There exists $\Lambda_0 \subseteq \Lambda$ that satisfies \eqref{set_condition_1} and \eqref{set_condition_2} such that, for any $z_0 \in \Lambda_0$ and $R_0>0$ sufficiently small, a neighborhood $U$ of $z_0$ exists with the property that
\begin{align}
\label{excess_decay}
\phi(z_0^{\prime}; r) \lesssim r^{2\alpha}
\end{align}
holds uniformly with respect to $z_0' \in U$ and for any $0<r < R_0$. Using the equivalence of Campanato and H\"{o}lder spaces (see, e.g., \cite[Proposition 1.1]{GS_82} or \cite[Theorem 3.1]{PS_65}) this implies the $C_{\loc}^{0,\alpha}(\Lambda_0)$-regularity of $u$.

The method for obtaining \eqref{excess_decay} for two sufficiently small radii $0<r\leq R$ is to view $u$ as a perturbation of the weak solution $\bar{u}$ of the frozen system 
\begin{align}
\label{MS_system_frozen}
\begin{split}
\partial_t \bar{u} - \nabla \cdot A( (u)_{z_0,R} ) \nabla \bar{u} & = f(u) \quad \quad \textrm{in}  \quad \Cc_{R/8}(z_0),\\
\bar{u}& = u \quad \quad \quad  \textrm{ on} \quad \partial^P \Cc_{R/8}(z_0).
\end{split}
\end{align}
We remark that the radius $``R/8"$ is used here for technical reasons that will become clear below. 

The issue of the solvability of \eqref{MS_system_frozen} can easily be settled using the (glued) entropy structure of \eqref{cross_diffusion_system}. In particular, let $B = \sqrt{h^{\prime\prime}_{\epsilon}((u)_{z_0,R})} \in \R^{n \times n}$ (constant positive definite matrix). 
We define the space $V$ as
\begin{align*}
V := \begin{cases}
\R^n & \mbox{if \hyperlink{C1}{(\textbf{C1})} holds}\\
\{ v\in\R^n~:~\sum_{i = 1}^n B_{ii}^{-1}v_i = 0 \}
& \mbox{if \hyperlink{C1p}{(\textbf{C1$^{\prime}$})} holds}.
\end{cases}
\end{align*}
Decomposing $\bar{u} = u + B^{-1}v$ and using the equations \eqref{cross_diffusion_system} and \eqref{MS_system_frozen} (both left-multiplied with $B$) we obtain a linear evolution equation for $v$:
\begin{align}
\label{MS_system_frozen_II}
\begin{split}
\partial_t v - \nabla \cdot \mathcal{A}\nabla v & = 
F(u,\nabla u) \quad \quad \textrm{in}  \quad \Cc_{R/8}(z_0),\\
v& = 0 \quad \quad \quad  \textrm{ on} \quad \partial^P \Cc_{R/8}(z_0),
\end{split}
\end{align}
where $F(u,\nabla u) = B \left[ \nabla \cdot ( A(u) - A( (u)_{z_0,R} )) \nabla u  \right]$ and 
$\mathcal{A} =B A( (u)_{z_0,R} )B^{-1}$.
Thanks to Definition~\ref{weak_solution} and \hyperlink{H4}{(\textbf{H4})}, $F(u,\nabla u)\in L^2(0,T; H^{-1}(\Cc_{R/8}(z_0);V))$, while the
(constant) matrix $\mathcal{A}$ is positive definite on $V$ due to either \hyperlink{C1}{(\textbf{C1})} or \hyperlink{C1p}{(\textbf{C1$^{\prime}$})}. At this point \cite[Thr.~23.A]{Z_90}
yields the existence of a unique weak solution $v\in L^2(0,T; H^1_0(\Cc_{R/8}(z_0);V))$ to \eqref{MS_system_frozen_II}, meaning that $\bar{u} = u + B^{-1}v\in L^2(0,T; H^1(\Cc_{R/8}(z_0); \R^n))$ is the unique solution to \eqref{MS_system_frozen}. 

In order to transfer regularity from $\bar{u}$ onto $u$, we must first show that $\bar{u}$ is sufficiently regular. To see this it is necessary to show that $\bar{u}$ satisfies a Caccioppoli inequality. In particular, we show that:

\begin{lemma}[Caccioppoli inequality for solutions of (\ref{MS_system_frozen})] 
\label{linear_cacc} We adopt the assumptions of Theorem \ref{Theorem_1}. For a point $z_0 \in \Lambda$, let  $\bar{u}$ be the weak solution of the frozen system \eqref{MS_system_frozen} on $\Cc_{R/8}(z_0)$. Then, for $z_0^{\prime} \in \Cc_{R/8}(z_0)$ and $ r>0 $ such that $\Cc_{2 r }(z_0^{\prime}) \subset  \Cc_{R/8}(z_0)$, we have that
\begin{align}
\label{linear_cacc_eq}
 \int_{\Cc_r(z^{\prime}_0)}  |\nabla \bar{u} |^2 \, \textrm{d}z
 \lesssim  \frac{1}{r^2} \int_{\Cc_{2r}(z^{\prime}_0) } |\bar{u} - b|^2 \, \textrm{d}z + r^{d+4} \|f (u)\|_{L^{\infty}(\Cc_{2r}(z^{\prime}_0))}^2
\end{align}
for any $b \in \mathbb{R}^n$.

\end{lemma}
\noindent We remark that in the proof of Lemma \ref{linear_cacc}, we replace the role of the energy estimate in the proof of the classical Caccioppoli estimate for parabolic systems by a ``frozen-in'' entropy estimate.

Using Lemma \ref{linear_cacc}, we can then derive the required interior regularity estimates for $\bar{u}$, which we give in the below corollary. 

\begin{corollary}[Interior regularity estimates for solutions of (\ref{MS_system_frozen})] 
\label{constant_coeff}
We adopt the assumptions of Theorem \ref{Theorem_1}. Let $z_0 \in \Lambda$ and $\bar{u}$ be the weak solution of the frozen system \eqref{MS_system_frozen} on $\Cc_{R/8}(z_0)$. Then, for any point $z_0^{\prime} \in \Cc_{R/8} (z_0)$ and radii $0 < r < \tilde{R}<1$ such that $\Cc_{\tilde{R}}(z^{\prime}_0) \subset \Cc_{R/8}(z_0)$, we find that 
\begin{align}
\label{constant_coeff_reg_1}
\int_{\Cc_r(z^{\prime}_0)}|\nabla \bar{u}|^2 \dd z 
 \lesssim   \Big(\frac{r}{\tilde{R}}\Big)^{d+2}  \int_{\Cc_{\tilde{R}}(z^{\prime}_0)} |\nabla \bar{u}|^2 \, \dd z + \tilde{R}^{d+4}
\end{align}
and 
\begin{align}
\label{constant_coeff_reg_2}
\int_{\Cc_r(z^{\prime}_0)} \Big|\nabla \bar{u} - (\nabla \bar{u})_{z_0^{\prime}, r}  \Big|^2 \dd z
 \lesssim   \Big(\frac{r}{\tilde{R}}\Big)^{d+4}  \int_{\Cc_{\tilde{R}}(z^{\prime}_0)} \Big| \nabla \bar{u}-(\nabla \bar{u})_{z^{\prime}_0, \tilde{R}} \Big|^2 \dd z + \tilde{R}^{d+4}.
\end{align}
\end{corollary}

We now indicate how to transfer the regularity from $\bar{u}$ onto $u$ in order to obtain \eqref{excess_decay}. First, notice that we may assume that $r< R/16$. Then, the triangle inequality allows us to write
\begin{align}
\label{triangle}
\int_{\Cc_{2r}(z_0)} |\nabla u |^2 \dd z \lesssim \int_{\Cc_{2r}(z_0)} |\nabla \bar{u}|^2 \dd z + \int_{\Cc_{R/8}(z_0)} |\nabla \bar{v}|^2 \dd z,
\end{align}
where we have introduced the error $\bar{v}:= \bar{u} - u$. Using \eqref{constant_coeff_reg_1} from Corollary \ref{constant_coeff} and the additional observation that 
\begin{align}
\label{energy_estimate_use}
\int_{\Cc_{R/8}(z_0)} |\nabla \bar{u}|^2 \dd z \lesssim  \int_{\Cc_{R/8}(z_0)} |\nabla u|^2 \dd z,
\end{align}
which is shown in Section \ref{Proof_Theorem_1}, we find that \eqref{triangle} becomes 
\begin{align}
\label{triangle_2}
\int_{\Cc_{2r}(z_0)} |\nabla u |^2 \dd z \lesssim   \Big( \frac{r}{R} \Big)^{d+2}\int_{\Cc_{R/8}(z_0)} |\nabla u|^2 \dd z + R^{d+4} + \int_{\Cc_{R/8}(z_0)} |\nabla \bar{v}|^2 \dd z.
\end{align}

Let us now treat the term in \eqref{triangle_2} involving $\overline{v}$. To do this, we notice that $\bar{v}$ is a weak solution of 
\begin{align}
\label{MS_system_frozen_aux}
\begin{split}
\partial_t B \bar{v} - \nabla \cdot B A( (u)_{z_0,R} ) \nabla \bar{v} & = \nabla \cdot  B (A(  (u)_{z_0,R}) - A(u)) \nabla  u  \qquad  \qquad \textrm{ in}  \quad \Cc_{R/8}(z_0),\\
\bar{v}& = 0 \quad \quad \quad  \quad \qquad \quad \quad \quad  \quad \quad \quad \quad \quad  \quad \, \, \quad\textrm{on} \quad \partial^P \Cc_{R/8}(z_0),
\end{split}
\end{align}
where we again use $B = \sqrt{h^{\prime\prime}_{\epsilon}((u)_{z_0,R})} \in \R^{n \times n}$. To obtain the desired estimate for $\bar{v}$ we then test \eqref{MS_system_frozen_aux} with $B \bar{v}$ and use the properties of the glued entropy, to write
\begin{align}
\label{error_estimate_frozen}
\begin{split}
\int_{\Cc_{R/8}(z_0)} |\nabla \bar{v} |^2 \dd z  \lesssim  \Big( \int_{\Cc_{R/8}(z_0)}   |\nabla  u|^p  \dd z  \Big)^{\frac{2}{p}} \Big(  \int_{\Cc_{R/8}(z_0)}  |A((u)_{z_0,R}) - A(u)|^{\frac{2p}{p-2}} \Big)^{\frac{p-2}{p}},
\end{split}
\end{align}
for $p>2$. 

The right-hand side of \eqref{error_estimate_frozen} is exactly analogous to the classical setting --see \cite{GS_82}.  And, just as in the classical setting, to handle \eqref{error_estimate_frozen} we now rely on the solution $u$ of \eqref{cross_diffusion_system} satisfying a reverse H\"{o}lder inequality. In particular, in Section \ref{reverse_holder_section} we show that:

\begin{prop}[Reverse H\"{o}lder inequality for solutions of (\ref{cross_diffusion_system})]
\label{reverse_holder}
We adopt the assumptions of Theorem \ref{Theorem_1}. Fix $ z_0 \in \Lambda$ and $R>0$ such that $\Cc_{4R}(z_0) \subset \Lambda$. Then, there exists $p >2$ such that $\nabla u \in L^p(\Cc_R(z_0))$ and
\begin{align}
\label{reverse_holder_relation}
\left( \fint_{\Cc_R(z_0)} |\nabla u|^p \dd z \right)^{\frac{1}{p}} \lesssim \left( \fint_{\Cc_{4R}(z_0)} |\nabla u|^2 \dd z \right)^{\frac{1}{2}} + R.
\end{align}
\end{prop}
\noindent Thereby, by choosing the appropriate $p>2$ in \eqref{error_estimate_frozen} we find that 
\begin{align}
\label{error_estimate_frozen_3}
 \int_{\Cc_{R/8}(z_0)} |\nabla \bar{v} |^2 \dd z \lesssim   \Big( \int_{\Cc_{R/2}(z_0)}   |\nabla  u|^2  \dd z  + R^{d+4} \Big) \Big(  \fint_{\Cc_{R/8}(z_0)}  |A((u)_{z_0,R}) - A(u)|^{\frac{2p}{p-2}} \Big)^{\frac{p-2}{p}} 
\end{align}

As we will justify in Section \ref{Proof_Theorem_1}, by combining \eqref{triangle_2} and \eqref{error_estimate_frozen_3} and using that $z_0 \in \Lambda_0$ with the characterization \eqref{set_condition_1}, we obtain \eqref{excess_decay} with $z_0^{\prime} = z_0$ for $R_0>0$ small enough. We then argue that \eqref{excess_decay} holds uniformly in a neighborhood of $z_0$.

\section{Proof of Proposition \ref{glued_entropy}: Entropy condition for the glued entropy}
\label{proof_Prop_1}

Our argument for Proposition \ref{glued_entropy} is where we formally capitalize on the intuition that we have given in Section \ref{construction} (and have encoded in the conditions \hyperlink{H3}{(\textbf{H3})} and \hyperlink{H3p}{(\textbf{H3$^{\prime}$})}). Here is the argument:

\begin{proof}
We first prove $(i)$ and assume that \hyperlink{H1}{(\textbf{H1})}, \hyperlink{H2}{(\textbf{H2})}, and \hyperlink{H3}{(\textbf{H3})} are satisfied. Let $y\in \dom$ be arbitrary.  For $\epsilon>0$ we define the sets:
\begin{align*}
\Ss_{\epsilon}(y) := \left\{i\in \Ss_1 \, : \, y_i \geq 2 \epsilon \right\}  \cup \Ss_2
\qquad \textrm{and} \qquad 
\Ss_{\epsilon}^c(y) := \left\{i\in \Ss_1 \, : \, y_i < 2 \epsilon \right\},
\end{align*}
where we have used the convention \eqref{case_1_S}. By the assumption \hyperlink{H2}{(\textbf{H2})} and the definitions \eqref{defn_glued_i} and \eqref{defn_glued} of $h_{\epsilon}$, the statement of the proposition is trivially true for all $y\in\mathcal D$ such that $\Ss_\epsilon(y)=\Ss$. 
Therefore, throughout this argument we will assume that $\Ss_\epsilon^c(y)\neq\emptyset$.

Fixing an arbitrary $\kappa >0$, we notice that by \hyperlink{H1}{(\textbf{H1})} there exists $\beta_\kappa>0$ such that, for any $i \in \Ss_1$, 
\begin{align}
\inf_{y_i \in (0,\beta_{\kappa}]} |h_i''(y_i)|\geq \kappa.
\label{lb.d2s}
\end{align}
Also, notice that \eqref{hp.A} of \hyperlink{H3}{(\textbf{H3})} implies
	\begin{align}\label{convergence.A}
\max_{i \in \Ss_1,j = 1, \ldots, n} |A_{ij}(y)-a_i(y)\delta_{ij}|  |h^{\prime \prime}_{\epsilon,i} (y)| \lesssim 1,
	\end{align}
where we have used that $h^{\prime \prime}_{\epsilon,i}(y) \lesssim h^{\prime \prime}_i(y)$. The latter observation follows from \eqref{hessian} in conjunction with the polynomial blow-up of each $h_i^{\prime \prime}$ that we have assumed in Case 1 of \hyperlink{H1}{(\textbf{H1})}.

Now, fix $\epsilon< \beta_{\kappa}/2$ and for $\rho\in \mathbb R^n$ define $\hat \rho \in\R^n$ as
	\[
	\hat \rho = (\hat{\rho}_1,\ldots,\hat{\rho}_n) \qquad \textrm{with} \qquad 
	\hat \rho_i := \begin{cases}
	\rho_i,& i\in \Ss_{\epsilon}(y),\\ 0, & i\in \Ss_{\epsilon}^c(y).
	\end{cases}\]
To show \eqref{prop_1_result}, we use the decomposition
\begin{align}
\label{decompose}
\rho \cdot h_\epsilon''(y)A(y)\rho  = (\rho-\hat \rho)\cdot h_\epsilon''(y)A(y)\rho + \hat \rho \cdot h_\epsilon''(y)A(y)\rho
\end{align}
and bound the two terms on the right-hand side separately. 

Starting with the first term of \eqref{decompose}, we write
	\begin{align}
	\label{calc_1}
	\begin{split}
	(\rho-\hat \rho)\cdot h_\epsilon''(y)A(y)\rho
	&=\sum_{i\in\Ss_{\epsilon}^c(y)}\sum_{j=1}^n \rho_i h''_{\epsilon,i}(y_i)A_{ij}(y)\rho_j	\\
	& = \sum_{i\in\Ss_{\epsilon}^c(y)}\rho_i h_{\epsilon,i}''(y_i)a_i(y)\rho_i+\sum_{i\in\Ss_{\epsilon}^c(y)}\sum_{j=1}^n \rho_i h_{\epsilon,i}''(y_i)(A_{ij}(y)-a_i(y)\delta_{ij})\rho_j.
	\end{split}
	\end{align}
By \eqref{hessian} and \hyperlink{H1}{(\textbf{H1})}, we have that $h_{\epsilon,i}''(y_i) \geq h_i''(2 \epsilon) \geq \kappa$ for $i\in\Ss_\epsilon^c(y)$. Using \eqref{mua} of  \hyperlink{H3}{(\textbf{H3})} and \eqref{convergence.A} it follows that
	\begin{align}
	\label{est.1}
	\begin{split}
	(\rho-\hat \rho)\cdot h_\epsilon''(y)A(y)\rho&\geq 
	\mu \kappa \sum_{i\in\Ss_{\epsilon}^c(y)}|\rho_i|^2 
	-C\Big(\sum_{i\in\Ss_{\epsilon}^c(y)}  |\rho_i| \Big) \Big(\sum_{j=1}^n |\rho_j| \Big)
	\\  &\geq \mu \kappa |\rho-\hat \rho|^2-  C n |\rho||\rho-\hat \rho|
	\\ &\geq (\mu \kappa- Cn) |\rho-\hat \rho|^2
	- Cn|\hat \rho||\rho-\hat \rho|.
	\end{split}
	\end{align}

We then treat the second term on the right-hand side of \eqref{decompose}. Using \eqref{entropy} of \hyperlink{H1}{(\textbf{H1})}, we write	
\begin{align*}
\begin{split}
  \hat \rho \cdot h''(y)A(y) \rho  = & \hat \rho \cdot h''(y)A(y)\hat \rho
+ \sum_{i \in \Ss_2} \sum_{j\in \Ss^c_{\epsilon}(y)} \rho_i h_i''(y)A_{ij}(y) \rho_j \\
& + \sum_{i\in\Ss_{\epsilon} (y) \setminus \Ss_2}\sum_{j\in \Ss^c_{\epsilon}(y)} \rho_i h_i''(y)(A_{ij}(y) - a_i(y)\delta_{ij})\rho_j.
\end{split}
	\end{align*}
Using \eqref{lb.d2hA} of \hyperlink{H2}{(\textbf{H2})} for the first term, $h''_i$ being bounded on $\overline{\dom}$ for $i \in \Ss_2$ for the second, and \eqref{hp.A} of \hyperlink{H3}{(\textbf{H3})} for the last summand above, we deduce
\begin{align}\label{est.2}
\hat \rho \cdot h''(y)A(y) \rho
\geq \lambda |\hat \rho|^2 -Cn|\rho-\hat \rho||\hat \rho|.
\end{align}
By \eqref{hessian} and the definition of $\Ss_{\epsilon}(y)$, we conclude that 
	\begin{align}
	\label{est.3}
	\hat \rho \cdot h_\epsilon''(y)A(y)\rho 
	=\hat \rho \cdot h''(y)A(y)\rho \geq 
	 \lambda |\hat \rho|^2 -Cn |\hat \rho| |\rho-\hat \rho| .
	\end{align}
	
Combining \eqref{decompose}, \eqref{est.1}, and \eqref{est.3} and using Young's inequality leads to
	\begin{align*}
	\rho \cdot h_\epsilon''(y)A(y) \rho \geq \frac{\lambda}{2} |\hat \rho|^2+\left(\mu \kappa- C(n, \lambda)\right)|\rho-\hat \rho|^2.
	\end{align*}
Choosing $\kappa>0$ sufficiently large yields \eqref{prop_1_result} for $y \in \dom$ --\eqref{prop_1_result} holds for $y \in \overline{\dom}$ since $h_{\epsilon}^{\prime\prime} A \in C^0(\overline{\dom})$.

\medskip

We now move-on to the ``explicit'' volume-filling case: The argument for $(ii)$ is almost the same as for $(i)$. The only minor difference is the assumption \hyperlink{H2p}{(\textbf{H2}$^{\prime}$)} (as opposed to \hyperlink{H2}{(\textbf{H2})}), which results in the analogue of \eqref{est.2} being 
\begin{align}
\label{case_2_1}
\begin{split}
\hat \rho \cdot h''(y)A(y) \rho
&\geq \lambda |\hat{\rho}|^2 - \delta\Big( \sum_{i=1}^n  \hat{\rho}_i \Big)^2 -Cn|\rho-\hat \rho||\hat \rho|\\
& = \lambda |\hat{\rho}|^2 - \delta \Big(\sum_{i=1}^n  \hat{\rho}_i \Big) \Big(\sum_{j=1}^n (\hat{\rho}_j - \rho_j ) \Big) -Cn|\rho-\hat \rho||\hat \rho|\\
& \geq \lambda |\hat{\rho}|^2  -(C+ \delta)n |\rho-\hat \rho||\hat \rho|,
\end{split}
\end{align}
where we have used that $\rho \in \Xi_0$ and $y \in \dom \cap \Xi_1$. Using \eqref{case_2_1} within the argument from part $(i)$ finishes the proof.  

\end{proof}

\section{Proofs of the Main Results: Theorems \ref{theorem_MS} -- \ref{Theorem_examples_5}}
\label{proofs_main_sec}

\subsection{``Implicit'' volume-filling systems: Proof of Theorem \ref{theorem_MS}}
\label{implicit_systems}

By ``implicit'' volume-filling system, we mean a system of the form \eqref{Maxwell_Stefan}, which requires the  inversion of a flux-gradient relation in order to be written in the form of \eqref{cross_diffusion_system}. For our partial regularity theory to be applicable it is necessary that this inversion be possible --in particular, we are able to treat systems of the form \eqref{Maxwell_Stefan} with a flux-gradient relation \eqref{flux_gradient_condition}, satisfying \eqref{volume_filling_constraint}, \eqref{ic_volume_filling}, and \eqref{reproduction_rates}, if the following condition is satisfied:

\begin{itemize}[leftmargin=.5in]
\item[\hyperlink{H0}{(\textbf{H0})}]  For any $y \in \dom$, $M(y) |_{\Xi_0} =: \tilde{M}(y)$ ($\Xi_0$ defined in \eqref{defn_subspace}) is invertible and $\| M(y) \|_2 \lesssim 1$ uniformly for $y \in \overline{\dom}$. Here, $\| \cdot \|_2$ denotes the largest singular value.
\end{itemize}

\noindent Contingent to \hyperlink{H0}{(\textbf{H0})}, an ``implicit'' system \eqref{Maxwell_Stefan} with flux-gradient relation \eqref{flux_gradient_condition} can be re-written as
\begin{align}
\label{rewritten_implicit}
\partial_t u - \nabla \cdot (\tilde{M}(u))^{-1} \nabla u = f(u), \quad \text{in} \quad \Omega\times (0,T),
\end{align}
which is, of course, now in ``explicit'' form. 

To ensure the applicability of our partial regularity theory we must then have access to a glued entropy in the sense of Definition \ref{glued_def} for the system \eqref{rewritten_implicit}. Towards this, we must make the following assumptions on $M$ in the flux-gradient relation \eqref{flux_gradient_condition}:

\smallskip

\paragraph{Conditions for ``implicit'' volume-filling systems:}

\begin{itemize}[leftmargin=.5in]
\item[\hyperlink{H2pp}{(\textbf{H2}$^{\prime\prime}$)}] \hyperlink{H2p}{(\textbf{H2}$^{\prime}$)} should hold for $A$ replaced by $M$.\\
\item[\hyperlink{H3pp}{(\textbf{H3}$^{\prime\prime}$)}]\hyperlink{H3p}{(\textbf{H3}$^{\prime}$)} should hold for $A$ replaced by $M$ and, accordingly, we ask for the existence of functions $m_i \in C^0(\overline{\dom})$ (instead of $a_i$).
\end{itemize}
\medskip

We obtain the following result:

\begin{prop}
Assume a system of the form \eqref{Maxwell_Stefan} with a flux-gradient relation \eqref{flux_gradient_condition}, satisfying \eqref{volume_filling_constraint}, \eqref{ic_volume_filling}, and \eqref{reproduction_rates}, and which satisfies \hyperlink{H0}{(\textbf{H0})}. Furthermore, assume that \hyperlink{H1}{(\textbf{H1})}, \hyperlink{H2pp}{(\textbf{H2}$^{\prime\prime}$)}, and \hyperlink{H3pp}{(\textbf{H3}$^{\prime\prime}$)} hold. Then, there exist $\epsilon > 0$ and $\lambda>0$ such that 
	\begin{align}
	\label{prop_3_result}
 \rho \cdot h_\epsilon''(y)A(y) \rho \geq \lambda |\rho|^2,
	\end{align}
	for any $\rho \in \Xi_0$ and $y \in \overline{\dom} \cap \Xi_1$, with $A(\cdot) := (\tilde{M}(\cdot))^{-1}$.\\
\end{prop}

\begin{proof}
By the argument for $(ii)$ of Proposition \ref{glued_entropy}, we may find $\epsilon>0$ such that, for any $y \in \overline{\dom}\cap \Xi_1$ and $\rho \in \Xi_0$, the relation
\begin{align}
\label{intermediate_glued_ent_1}
\rho \cdot h^{\prime\prime}_{\epsilon}(y) M(y) \rho \geq \lambda |\rho|^2 
\end{align}
holds for some $\lambda>0$. Then, let $\xi = A(y) \rho$ and notice that $\xi \in \Xi_0$.  Applying \eqref{intermediate_glued_ent_1} we obtain
\begin{align}
\label{oct_28_1}
 \rho \cdot  h_{\epsilon}^{\prime \prime}(y) A(y) \rho = A(y) \rho \cdot h_{\epsilon}^{\prime \prime}(y) \rho  = \xi \cdot h^{\prime \prime}_{\epsilon}(y) M (y) \xi \geq \lambda |\xi|^2,
\end{align} 
where we have additionally used the definition of $A(y)$ and that $h^{\prime\prime}_{\epsilon}(y)$ is symmetric. To finish our argument for \eqref{prop_3_result}, we write
\begin{align}
\label{oct_28_2}
|\rho|^2 = |M(y) \xi|^2 \leq \|M(y) \|^2_2 |\xi|^2 \lesssim |\xi|^2,
\end{align}
where we have used \hyperlink{H0}{(\textbf{H0})}. Combining \eqref{oct_28_1} and \eqref{oct_28_2} yields \eqref{prop_3_result}.
\end{proof}

\smallskip

\paragraph{Argument for Theorem \ref{theorem_MS}} We now consider the Maxwell-Stefan model as presented in Example \ref{MS_model}: Notice that for any $y \in \dom $, using the symmetry of the interspecies diffusion coefficients, it can easily be checked that $\rm{Im}(M_{\rm{MS}}(\textit{y})) \subseteq \Xi_0$. It is shown in \cite{JS_2012} that $\tilde{M}_{\rm{MS}}(y) = M_{\rm{MS}}(y)|_{\Xi_0}$ is invertible, by which \eqref{Maxwell_Stefan} with \eqref{flux_gradient_condition} and \eqref{volume_filling_constraint} becomes
\begin{align}
\label{rewritten_MS}
\partial_t u - \nabla \cdot (\tilde{M}_{\rm{MS}}(u))^{-1} \nabla u = f(u), \quad \text{in} \quad \Omega\times (0,T),
\end{align}
as in \eqref{rewritten_implicit}. Following the calculations in \cite[Section 2]{JS_2012}, for $u^{\prime} = (u_1, \ldots, u_{n-1})^T$,  defining $M_{0}(u^{\prime})$ via 
\begin{align}
(M_0(u^{\prime}))_{ij} = 
\begin{cases}
\sum_{k=1, k \neq i}^{n-1}\left( \frac{1}{D_{ik}} - \frac{1}{D_{in}} \right) u^{\prime}_k + \frac{1}{D_{i n}} & \text{if } i =j, \, i,j = 1, \ldots, n-1,\\
- \left( \frac{1}{D_{ij}} - \frac{1}{D_{in}} \right)u^{\prime}_i & \text{if } i \neq j, \, i,j = 1, \ldots, n-1,
\end{cases}
\end{align}
the partial inverse $(\tilde{M}_{\rm{MS}}(y))^{-1}$ is obtained as 
\begin{align}
\label{inverse}
(\tilde{M}_{\rm{MS}}(y))^{-1} = X
\left[\begin{array}{cc}
(M_0(y^{\prime}))^{-1} & 0\\
0 & 0 
\end{array}\right]
X^{-1}, \quad \text{where } X = \rm{Id}_n - \left[ \begin{array}{r} 0\\ \vdots \\0\\1   \end{array} \right] \otimes \left[ \begin{array}{r} 1\\ \vdots \\1\\0   \end{array} \right].
\end{align}
The weak solutions of the Maxwell-Stefan system that are obtained in \cite{JS_2012} are given by $u = X (u^{\prime}, 1)^T$, where $u^{\prime}$ solves 
\begin{align}
\label{reduced_MS}
\partial_t u^{\prime} - \nabla \cdot (M_0(u^{\prime}))^{-1} \nabla u^{\prime} = f^{\prime}(u^{\prime})
\end{align}
for $(f^{\prime}(u^{\prime}),0)^T = X^{-1} f(u^{\prime})$ --by construction the components of $u^{\prime}$ are nonnegative and $\sum_{i=1}^{n-1} u_i \leq 1$. Notice that we have slightly abused notation by writing $``f(u^{\prime})"$, by which we mean $f((u^{\prime},1 - \sum_{i=1}^{n-1}u^{\prime}_{i})^T)$.\\

We now give the proof of Theorem \ref{theorem_MS}:

\begin{proof}[Proof of Theorem \ref{theorem_MS}]
We first remark that it follows from the proofs that Theorems \ref{Theorem_1} and \ref{higher_reg} may be applied when \hyperlink{H1}{(\textbf{H1})}, \hyperlink{H4}{(\textbf{H4})},  \hyperlink{H5}{(\textbf{H5})}, and \eqref{volume_perserving} hold, and there exists a glued entropy (see Definition \ref{glued_def}, the ``explicit'' volume-filling case). For the application of Theorem \ref{higher_reg} we must, additionally, have that $A \in C^{0,\sigma}_{\rm{loc}}(\dom)$. We, therefore, now check that \hyperlink{H0}{(\textbf{H0})}, \hyperlink{H1}{(\textbf{H1})}, \hyperlink{H2pp}{(\textbf{H2$^{\prime\prime}$})}, \hyperlink{H3pp}{(\textbf{H3$^{\prime\prime}$})}, \hyperlink{H5}{(\textbf{H5})}, and \eqref{volume_perserving} hold, and that $A_{\rm{MS}} = (\tilde{M}_{\rm{MS}})^{-1} \in C^{0,\sigma}(\overline{\dom})$ for any $\sigma \in (0,1)$. 

The condition \hyperlink{H0}{(\textbf{H0})} has been discussed above and \hyperlink{H5}{(\textbf{H5})} holds by assumption.  \hyperlink{H1}{(\textbf{H1})} follows immediately from \eqref{B_entropy} and \hyperlink{H2pp}{(\textbf{H2$^{\prime \prime}$})} is \eqref{hypocoercive_MS}. We will see that \hyperlink{H3pp}{(\textbf{H3$^{\prime\prime}$})} holds with 
\begin{align*}
m_i(y) = \frac{D_i + \sum^n_{j=1, j \neq i}( \prod_{k =1, k \neq j, k \neq i}^n D_{ik}  - D_i ) y_j }{\prod^n_{j =1, j \neq i}D_{ij}},
\end{align*}
where we set
\begin{align*}
D_i := \min_{j \in \{1, \ldots, n\}, j \neq i} \prod_{k =1, k \neq j, k \neq i }^n D_{ik}.
\end{align*} 
The condition corresponding to \eqref{mua} is clearly satisfied thanks to the positivity of $\dom$ and the definition of the $D_i$. To check \eqref{hp.M} we notice that for $y \in \dom \cap \Xi_1$, we have that 
\begin{align*}
M_{\text{MS},ii} (y) = \frac{D_i (1-y_i) + \sum^n_{j =1, j \neq i}( \prod^n_{k=1, k \neq j, k \neq i} D_{ik}  - D_i ) y_j}{\prod^n_{j =1, j \neq i}D_{ij}},
\end{align*}
whereby it is easy to see that \eqref{hp.M} holds. The volume-filling condition \eqref{volume_perserving} has been observed already following Example \ref{MS_model}.

We must still check that $A_{\rm{MS}} = (\tilde{M}_{\rm{MS}})^{-1} \in C^{0,\sigma}(\overline{\dom})$. For this, we recall \eqref{inverse} and use that in \cite[Lemma 5]{JS_2012} it is shown that the spectrum $\sigma(M_0) \subset [ \delta, \Delta )$, where $\delta, \Delta >0$ are finite constants. It is then shown via Cramer's rule that $M_0^{-1}$ is uniformly bounded on $\overline{\dom}$. In a similar way it can also be shown that $(M_0)^{-1}_{ij} \in C^{0,\sigma}(\overline{\dom})$ for $i, j = 1, \ldots, n-1$. The desired regularity of $A_{\rm{MS}}$ then follows immediately from \eqref{inverse}. 

\end{proof}

\subsection{``Explicit'' volume-filling systems: Proof of Theorem \ref{Theorem_examples_2}}
We apply Theorems \ref{Theorem_1} and \ref{higher_reg} as:
\begin{proof}
For the application of Theorems \ref{Theorem_1} and \ref{higher_reg}, we check that the conditions \hyperlink{H1}{(\textbf{H1})}, \hyperlink{H2p}{(\textbf{H2$^{\prime}$})}, \hyperlink{H3p}{(\textbf{H3$^{\prime}$})}, \hyperlink{H5}{(\textbf{H5})}, and \eqref{volume_perserving} hold, and that $A_{\rm{HB}} \in C^{0,\sigma}_{\rm{loc}}(\overline{\dom})$ for any $\sigma \in (0,1)$. The condition \hyperlink{H1}{(\textbf{H1})} holds thanks to \eqref{B_entropy} and \hyperlink{H2p}{(\textbf{H2$^{\prime}$})} can been seen to hold via the discussion following Example \ref{MS_model}, since $A_{\text{HB}}$ and $M_{\text{MS}}$ are the same (up to renaming constants).  \hyperlink{H3p}{(\textbf{H3$^{\prime}$})} has been shown in the proof of Theorem \ref{theorem_MS}. The condition \hyperlink{H5}{(\textbf{H5})} holds by assumption and that $A_{\rm{HB}} \in C^{0,\sigma}_{\rm{loc}}(\overline{\dom})$ holds is easily seen via \eqref{A_HB}. The volume-filling condition has been verified following Example \ref{HB_model}.

\end{proof}

\subsection{Non volume-filling systems: Proofs of Theorems \ref{Theorem_examples_3}, \ref{Theorem_examples_4}, and \ref{Theorem_examples_5}}

In each of these proofs we must check that we can apply Theorems \ref{Theorem_1} and \ref{higher_reg}. Here are the arguments:

\begin{proof}[Proof of Theorem \ref{Theorem_examples_3}]
For the application of Theorems \ref{Theorem_1} and \ref{higher_reg}, we check that the conditions \hyperlink{H1}{(\textbf{H1})}, \hyperlink{H2}{(\textbf{H2})}, \hyperlink{H3}{(\textbf{H3})}, and \hyperlink{H5}{(\textbf{H5})} hold, and that $A_{\rm{SKT}} \in C_{\rm{loc}}^{0,\sigma}(\overline{\dom})$ for any $\sigma \in (0,1)$. The condition \hyperlink{H1}{(\textbf{H1})} follows immediately from \eqref{entropy_SKT}. The condition \hyperlink{H2}{(\textbf{H2})} is \eqref{entropy_condition_SKT}. For $\hyperlink{H3}{(\textbf{H3})}$ we take 
\begin{align}
a_1(y)= \alpha_{10} + \alpha_{12} u_2 \quad \text{and} \quad a_2(y)= \alpha_{20} + \alpha_{21} u_1,
\end{align}
where \eqref{mua} is satisfied by our assumption that $\alpha_{i0}>0$ for $i = 1,2$ and since the weak solutions are nonnegative. That $A_{\rm{SKT}} \in C_{\rm{loc}}^{0,\sigma}(\overline{\dom})$ is easily seen via \eqref{A_SKT} and \hyperlink{H5}{(\textbf{H5})}  holds by assumption.
 \end{proof} 

\begin{proof}[Proof of Theorem \ref{Theorem_examples_4}]
For the application of Theorems \ref{Theorem_1} and \ref{higher_reg}, we check that the conditions \hyperlink{H1}{(\textbf{H1})}, \hyperlink{H2}{(\textbf{H2})}, \hyperlink{H3}{(\textbf{H3})}, and \hyperlink{H5}{(\textbf{H5})} hold, and that $A_{\rm{SC}} \in C^{0,\sigma}(\overline{\dom})$ for any $\sigma \in (0,1)$. The condition \hyperlink{H1}{(\textbf{H1})} holds, since the relevant entropy density is given by \eqref{B_entropy}. \hyperlink{H2}{(\textbf{H2})} follows from \eqref{entropy_structure_SC}. For \hyperlink{H3}{(\textbf{H3})} we take 
\begin{align}
a_1(y)= \frac{\mu_1}{1+ \mu_2 u_1 + \mu_1 u_2}  \quad \text{and} \quad a_2(y)= \frac{\mu_2}{1+ \mu_2 u_1 + \mu_1 u_2},
\end{align}
where \eqref{mua} is satisfied since $\mu_2, \mu_2 >0$ and the components of $u$ are nonnegative. That $A_{\rm{SC}} \in C^{0,\sigma}(\overline{\dom})$ can be verified via \eqref{A_semi} and \hyperlink{H5}{(\textbf{H5})} holds by assumption.
\end{proof}

\begin{proof}[Proof of Theorem \ref{Theorem_examples_5}] For the application of Theorems \ref{Theorem_1} and \ref{higher_reg}, we check that the conditions \hyperlink{H1}{(\textbf{H1})}, \hyperlink{H2}{(\textbf{H2})}, \hyperlink{H3}{(\textbf{H3})}, and \hyperlink{H5}{(\textbf{H5})} hold, and that $A_{\rm{HJ}} \in C_{\rm{loc}}^{0,\sigma}(\overline{\dom})$ for any $\sigma \in (0,1)$. The condition \hyperlink{H1}{(\textbf{H1})} follows from \eqref{entropy_HJ} and \hyperlink{H2}{(\textbf{H2})} from \eqref{entropy_cond_HJ}. For \hyperlink{H3}{(\textbf{H3})} we notice that $2 \in \Ss_2$, where we use the notation \eqref{case_1_S}. Since $1 \in \Ss_1$, we let 
\begin{align}
a_1(y)=  u_1 +1.
\end{align}
The condition \hyperlink{H5}{(\textbf{H5})} and that  $A_{\rm{HJ}} \in C_{\rm{loc}}^{0,\sigma}(\overline{\dom})$ can easily be seen via \eqref{A_HJ}.
\end{proof}

\section{Estimate of Poincar\'{e}-Wirtinger type for solutions of (\ref{cross_diffusion_system})}
\label{poincare}

Throughout our arguments, we will make repeated use of the following Poincar\'{e}-Wirtinger type inequality that is satisfied by solutions of \eqref{cross_diffusion_system}, and also of \eqref{MS_system_frozen}. An argument of the type we use below can be found in \cite[Lemmas 3 and 4]{S_81}.

\begin{lemma} \label{Poincare} Let $u$ be a weak solution of \eqref{cross_diffusion_system} such that $A(u) \lesssim 1$. Fix $z_0 \in\Lambda$ and $R>0$ such that $\Cc_{2R}(z_0) \subset \Lambda$. Then, we find that the relation
\begin{align}
\label{poincare_new_2}
\int_{\Cc_R(z_0)} |u - (u)_{z_0, R}|^2 \, \dd z  \lesssim R^2 \int_{\Cc_{2R}(z_0)} |\nabla u |^2 \dd z +R^{d+6} \|f(u)\|_{L^{\infty}(\Cc_{2R}(z_0))}^2
\end{align}
 holds.\\
\indent The estimate \eqref{poincare_new_2} also holds for weak solutions of the frozen system \eqref{MS_system_frozen}, as long as $C_{2R}(z_0)$ is contained in the domain where the system is defined.
\end{lemma}

\begin{proof} Let $t_1$ and $t_2 \in \Gamma_{2R}(t_0)$ such that $0< t_1 < t_2$ and $\mathds{1}_{(t_1, t_2)}$ denote the indicator function of the interval $(t_1, t_2)$. We will first show that
\begin{align}
\label{intermediate_1_2}
|(\tilde{u})_{x_0, R}(t_1) - (\tilde{u})_{x_0, R}( t_2)|^2 \lesssim R^{-d} \int_{\Cc_{2R}(z_0)} |\nabla u|^2 \dd z + R^{4}\|f(u)\|^2_{L^{\infty}(\Cc_{2R}(z_0))},
\end{align}
where we use the notation \eqref{weighted_averages}. For this, we test the system \eqref{cross_diffusion_system} with $\chi_{x_0,R}^2 \mathds{1}_{(t_1, t_2)}$ to obtain
\begin{align*}
\begin{split}
&  \int_{B_{2R}(x_0)} \chi_{x_0,R}^2 u \dd x \Big|_{t=t_2} -  \int_{B_{2R}(x_0)} \chi_{x_0, R}^2 u \dd x \Big|_{t=t_1} \\
 &  =-   \int_{t_1}^{t_2} \int_{B_{2R}(x_0)} [\textrm{Id} \,  \otimes \nabla \chi_{x_0,R}^2 ]: A(u) \nabla u  \dd x \dd t + \int_{t_1}^{t_2} \int_{B_{2R}(x_0)} \chi_{x_0,R}^2 f(u) \dd x \dd t.
 \end{split}
\end{align*}
Taking the absolute value of both sides and using the definition \eqref{weighted_averages},  we obtain 
\begin{align}
\label{P_april_1}
\begin{split}
& R^d |(\tilde{u})_{x_0, R}(t_2) - (\tilde{u})_{x_0, R}(t_1)|\\
& \lesssim
  \int_{t_1}^{t_2} \int_{B_{2R}(x_0)}|   [ \textrm{Id}  \otimes \nabla \chi_{x_0,R}^2 ] : A(u) \nabla u | \dd x \dd t + R^{d+2} \|f(u)\|_{L^{\infty}(\Cc_{2R}(z_0))}.
  \end{split}
\end{align}
We then apply H\"{o}lder's inequality and inject the properties of $\chi_R$ to write 
\begin{align*}
  \int_{t_1}^{t_2} \int_{B_{2R}(x_0)}| [  \textrm{Id}  \otimes \nabla \chi_{x_0,R}^2 ]: A(u) \nabla u | \dd x \dd t
  \lesssim 
  R^{\frac{d}{2}} \Big( \int_{t_1}^{t_2} \int_{B_{2R}(x_0)}| \nabla u |^2 \dd x \dd t \Big)^{\frac{1}{2}},
\end{align*}
which is combined with \eqref{P_april_1} to give
\begin{align}
\label{Aug_3_1}
 |(\tilde{u})_{x_0, R}(t_2) - (\tilde{u})_{x_0, R}(t_1)| \lesssim  R^{-\frac{d}{2}} \Big( \int_{t_1}^{t_2} \int_{B_{2R}(x_0)}| \nabla u |^2 \dd x \dd t \Big)^{\frac{1}{2}} + R^{2}  \|f(u)\|_{L^{\infty}(\Cc_{2R}(z_0))}.
\end{align} 
The relation \eqref{intermediate_1_2} follows.

We now show \eqref{poincare_new_2}. Using a slight variant of the standard Poincar\'{e}-Wirtinger inequality and \eqref{intermediate_1_2}, we write
\begin{align}
\label{P_june_1}
\begin{split}
& \int_{\Cc_R(z_0)} |u - (u)_{z_0,R}|^2 \dd z \\
&\leq  \int_{\Cc_{2R}(z_0)} |u - \fint_{\Gamma_{2R}(t_0)} ( \tilde{u})_{x_0,R}(t) \dd t |^2 \dd z  \\
& \leq \int_{\Cc_{2R}(z_0)}  |u - ( \tilde{u})_{x_0,R}|^2 \dd z   + \int_{\Cc_{2R}(z_0)}  | ( \tilde{u})_{x_0,R}-\fint_{\Gamma_{2R}(t_0)} ( \tilde{u})_{x_0,R}(t) \dd t   |^2 \dd z\\
& \lesssim  R^2 \int_{\Cc_{2R}(z_0)} |\nabla u |^2 \dd z +  R^{d} \int_{\Gamma_{2R}(t_0)}  \fint_{\Gamma_{2R}(t_0)}  | ( \tilde{u})_{x_0,R}(s) -  ( \tilde{u})_{x_0,R}(t)|^2 \dd t \dd s  \\
&  \lesssim R^2 \int_{\Cc_{2R}(z_0)} |\nabla u|^2 \dd z + R^{d+2} \Big( R^{-d} \int_{\Cc_{2R}(z_0)} |\nabla u|^2 \dd z + R^4 \|f(u)\|^2_{L^{\infty}(\Cc_{2R}(z_0))} \Big).
\end{split}
\end{align}

The same strategy as above can be applied to weak solutions of \eqref{MS_system_frozen}.

\end{proof}

\section{Argument for Proposition \ref{reverse_holder}: A reverse H\"{o}lder inequality for solutions of (\ref{cross_diffusion_system})} 
\label{reverse_holder_section}

We follow the outline of the proof of \cite[Theorem 2.1]{GS_82}, but within the framework of the glued entropy introduced above.  In particular, the proof of Proposition \ref{reverse_holder} relies on the following result:

\begin{prop} \label{technical} Let $Q \subset \R^d \times (0,\infty)$ be a bounded space-time domain. Let $g, h : Q \to \R$ be nonnegative functions, where $g \in L^q(Q)$ and $h \in L^r(Q)$ with $r>q>1$. Suppose that for any $z_0 \in Q$ and $R>0$ such that $\Cc_{4R}(z_0) \subset Q$ the estimate
\begin{align*}
\fint_{\Cc_R(z_0)} g^q \dd z \leq b \left\{  \Big( \fint_{\Cc_{4R}(z_0)} g \dd z \Big)^{q}  + \fint_{\Cc_{4R}(z_0)} h^q \dd z \right\}+ \gamma \fint_{\Cc_{4R}(z_0)} g^q \dd z
\end{align*}
holds for $\gamma >0$. Then, there exists a constant $\gamma_0 = \gamma_0(q,r,d)$ such that if $\gamma < \gamma_0$, then there exists $\delta>0$ such that $g \in L^p_{\rm{loc}}(Q)$ for $p \in [ q, q+\delta)$ and
\begin{align}
\Big( \fint_{\Cc_R(z_0)} g^p \dd z \Big)^{\frac{1}{p}} \leq c \Big\{  \Big( \fint_{\Cc_{4R}(z_0)} g^q \dd z \Big)^{\frac{1}{q}}  + \Big( \fint_{\Cc_{4R}(z_0)} h^p \dd z \Big)^{\frac{1}{p}}  \Big\}
\end{align}
for any $z_0 \in Q$ and $R>0$ such that $\Cc_{4R}(z_0) \subset Q$. The constant $c$ and $\delta>0$ depend on $b$, $q$, $r$, $\gamma$, and $d$ only.
\end{prop}

\noindent The proof of this result can be found for elliptic systems in \cite[Proposition 5.1]{GM_79}. The argument goes via a Calder\'{o}n-Zygmund cube decomposition and can be adapted to the parabolic setting by replacing Euclidean cubes by parabolic cubes. 

As we will see below, to be able to apply Proposition \ref{technical} we require two additional ingredients: First, a Caccioppoli-type estimate like that in Lemma \ref{linear_cacc}, but for solutions of \eqref{cross_diffusion_system}; and second, another estimate of Poincar\'{e}-Wirtinger type satisfied by solutions of \eqref{cross_diffusion_system}, different from that in Lemma \ref{Poincare}. We start with the Caccioppoli-type estimate:
  
\begin{lemma}
\label{nonlinear_cacc}  We adopt the assumptions of Theorem \ref{Theorem_1}. Fix $z_0 \in \Lambda$ and  $R>0$ such that $\Cc_{2R}(z_0) \subset \Lambda$. We show that 
\begin{align}
\label{nonlinear_cacc_eq}
 \int_{\Cc_R(z_0)}  |\nabla u|^2 \, \textrm{d}z
 \lesssim   \frac{1}{R^2}  \int_{\Cc_{2R}(z_0)} |u - (\tilde{u})_{x_0,  R}|^2 \, \textrm{d}z + R^{d+4} \|f(u)\|^2_{L^{\infty}(\Cc_{2R}(z_0))}.
\end{align}
\end{lemma}
\noindent As described in Section \ref{motivation}, the argument for Lemma \ref{nonlinear_cacc_eq} is essentially an entropy estimate for $ \int_{\Omega} h_\epsilon(u \, \vert  \, (\tilde{u})_{x_0, R}) \dd x$, but infused with ingredients usually used to prove the classical Caccioppoli inequality. Before moving on, we remark that it follows from Lemma \ref{nonlinear_cacc} that
\begin{align}
\label{nonlinear_cacc_eq_2}
 \int_{\Cc_R(z_0)}  |\nabla u|^2 \, \textrm{d}z
 \lesssim   \frac{1}{R^2}  \int_{\Cc_{2R}(z_0)} |u - (u)_{z_0,2R}|^2 \, \textrm{d}z+ R^{d+4} \|f(u)\|^2_{L^{\infty}(\Cc_{2R}(z_0))}.
\end{align}

Using a similar method, we can also prove the other estimate of Poincar\'{e}-Wirtinger type mentioned above:

\begin{lemma} \label{time_reg}  We adopt the assumptions of Theorem \ref{Theorem_1}. Fix $z_0 \in \Lambda$ and  $R>0$ such that $\Cc_{2R}(z_0) \subset \Lambda$. Then we find that
\begin{align}
\label{Poincare_new}
\sup_{t\in \Gamma_R(t_0)} \int_{B_R(x_0)} |u(t) - (\tilde{u})_{x_0,R}(t)|^2  \, \textrm{d} x \lesssim \int_{\Cc_{2R}(z_0)} |\nabla u |^2 \, \textrm{d} z +R^{d+4}\|f(u)\|_{L^{\infty}(\Cc_{2R}(z_0))}^2.
\end{align}
\end{lemma}

\noindent We remark that Lemma \ref{time_reg} is an analogue of \cite[Lemma 2.2]{GS_82}. 

\subsection{Proof of Lemma \ref{nonlinear_cacc}: A Caccioppoli-type estimate for solutions of (\ref{cross_diffusion_system})}

\begin{proof}
Let $\tau \in C^{\infty}(\R)$ such that $\tau \equiv 1$ on $\Gamma_R(t_0)$,  $\tau \equiv 0$ on the set $t \leq  t_0 - (2R)^2 $, and $| \tau'| \lesssim 1/R^2$. Defining the cut-off function $\eta = \chi_{x_0,R}  \tau$, with the notation \eqref{cut_off}, and using the definition \eqref{relative_entropy_ep}, we then write
	\begin{align}
	\label{time_deriv_calc}
	\begin{split}
		& \partial_t \big( h_\epsilon(u \vert   (\tilde{u})_{x_0, R}) \eta^2 \big)\\
	&=
	\partial_t h_\epsilon(u  \vert   (\tilde{u})_{x_0, R} )\eta^2+
	2h_\epsilon(u  \vert   (\tilde{u})_{x_0, R})\eta\partial_t \eta
	\\&=\eta^2 \partial_u h_\epsilon(u  \vert   (\tilde{u})_{x_0, R}) \cdot \partial_t u+ \eta^2 \partial_{v} h_\epsilon(u  \vert   (\tilde{u})_{x_0, R} ) \cdot \partial_t (\tilde{u})_{x_0, R}
	+ 2h_\epsilon(u   \vert  (\tilde{u})_{x_0, R})\eta\partial_t \eta
	\\&=\eta^2 (h'_\epsilon(u)-h'_\epsilon(  (\tilde{u})_{x_0, R})) \cdot \partial_tu  - \eta^2 h''_\epsilon( (\tilde{u})_{x_0, R}) (u -  (\tilde{u})_{x_0, R} )  \cdot \partial_t (\tilde{u})_{x_0, R}+
	2h_\epsilon(u  \vert  (\tilde{u})_{x_0, R})\eta\partial_t \eta.
	\end{split}
	\end{align}
After integrating this identity and using the definition of $\eta$, we obtain 
	\begin{align}
	\label{time_deriv}
	\begin{split}
	0&\leq \int_{B_{2R}(x_0)}h_\epsilon(u\vert  (\tilde{u})_{x_0, R} )\eta^2 \dd x \bigg|_{t=t_0}\\
	& =
	\int_{C_{2R}(z_0)}\partial_t (h_\epsilon(u\vert  (\tilde{u})_{x_0, R} )\eta^2) \dd z
	\\& =\int_{C_{2R}(z_0)}\eta^2 (h'_\epsilon(u)-h'_\epsilon( (\tilde{u})_{x_0, R} )) \cdot \partial_tu \dd z + 2\int_{C_{2R}(z_0)}h_\epsilon(u | (\tilde{u})_{x_0, R}  )\eta\partial_t \eta \, \dd z  \\
	& \qquad - \int_{C_{2R}(z_0)} \eta^2 h''_\epsilon(  (\tilde{u})_{x_0, R} ) (u -  (\tilde{u})_{x_0, R} )  \cdot \partial_t   (\tilde{u})_{x_0, R}  \dd z
	\\& =: \rom{1}+ \rom{2} + \rom{3}.
	\end{split}
	\end{align}
	
Here we have used that $\partial_t  (\tilde{u})_{x_0, R}  \in L^{1}(\Gamma_{2R}(t_0))$ --which can be seen by testing \eqref{cross_diffusion_system} with $\chi^2_{R} \mathds{1}_{(t, t_0)}$ as in the proof of Lemma \ref{Poincare}-- and \eqref{time_deriv_calc} in the form 
\begin{align}
\begin{split}
&\Big\|  \partial_t \int_{B_{2R}(x_0)} h_\epsilon(u\vert  (\tilde{u})_{x_0, R} )\eta^2  \dd x \Big\|_{L^1(\Gamma_{2R}(t_0))}   \\
 & \lesssim_R \| u -   (\tilde{u})_{x_0, R} \|_{L^2( \Gamma_{2R}(t_0); H^1(B_{2R}(x_0)))} \| \partial_t u  \|_{L^2(\Gamma_{2R}(t_0);H^{-1}(B_{2R}(x_0)))} \\
&\qquad  + \| \partial_t (\tilde{u})_{x_0, R}\|_{L^{1} (\Gamma_{2R}(t_0))}  \| u -   (\tilde{u})_{x_0, R} \|_{L^{\infty}(\Gamma_{2R}(z_0),  L^2(B_{2R}(x_0)))}\\
& \qquad \qquad + \| u -   (\tilde{u})_{x_0, R} \|_{L^2( \Gamma_{2R}(t_0), L^2(B_{2R}(x_0)))} < \infty,
\end{split}
\end{align}
where we have used the properties of the glued entropy. We remark that the above heuristic computation can be made formal with a standard approximation argument.

Let us now treat the terms on the right-hand side of \eqref{time_deriv} separately, starting with \textit{\rom{1}}. From \eqref{cross_diffusion_system} it follows:
	\begin{align}
	  \label{final_cacc_june_2}
	\begin{split}
	I&=-
	\int_{C_{2R}(z_0)}\nabla\big((h'_\epsilon(u)-h'_\epsilon( (\tilde{u})_{x_0, R} ))\eta^2\big) : A(u)\nabla u \dd z + \int_{C_{2R}(z_0)} \eta^2(h'_\epsilon(u)-h'_\epsilon( (\tilde{u})_{x_0, R} ))\cdot f(u) \dd z 
	\\&=-
	\int_{C_{2R}(z_0)}\eta^2\nabla h'_\epsilon(u) : A(u)\nabla u \dd z - 2
	\int_{C_{2R}(z_0)} \eta [(h'_\epsilon(u)-h'_\epsilon( (\tilde{u})_{x_0, R} )) \otimes  \nabla \eta ]: A(u)\nabla u \dd z\\
	& \qquad + \int_{C_{2R}(z_0)} \eta^2 (h'_\epsilon(u)-h'_\epsilon( (\tilde{u})_{x_0, R} )) \cdot f(u) \dd z\\
	& =: \rom{1}_1 +\rom{1}_2 + \rom{1}_3,
	 \end{split}
	 \end{align}
From \hyperlink{C1}{(\textbf{C1})} or \hyperlink{C1p}{(\textbf{C1$^{\prime}$})} with $\eqref{volume_perserving}$ we then obtain 
	 \begin{align}
	  \label{final_cacc_june}
	 I_1 &= -
	\int_{C_{2R}(z_0)}\eta^2  \nabla u:  h_\epsilon''(u) A(u) \nabla u \dd z \lesssim - \int_{C_{2R}(z_0)}\eta^2|\nabla u |^2 \dd z.
	\end{align}
The term $\rom{1}_2$ is treated using \hyperlink{C2}{(\textbf{C2})} and the boundedness of $A$ from \hyperlink{H4}{(\textbf{H4})} as
	 \begin{align}
	 \label{final_cacc_1}
	 \begin{split}
	| \rom{1}_2 |& \lesssim \int_{C_{2R}(z_0)}\eta|\nabla\eta | |u- (\tilde{u})_{x_0, R} | \, |\nabla u| \dd z \\
	& \lesssim \gamma  \int_{C_{2R}(z_0)}\eta^2|\nabla u|^2 \dd z + C(\gamma) \int_{C_{2R}(z_0)}|\nabla \eta|^2|u- (\tilde{u})_{x_0, R} |^2 \dd z,
	\end{split}
	 \end{align}
for any $\gamma >0$. 
For $\rom{1}_3$ we use \hyperlink{C2}{(\textbf{C2})} and Young's inequality to write
 \begin{align}
 	 \label{final_cacc_1_new}
	 \begin{split}
|\rom{1}_3| & \lesssim  \int_{C_{2R}} \eta^2 \Big(  \frac{1}{ R^2} |h'_\epsilon(u)-h'_\epsilon( (\tilde{u})_{x_0, R} ) |^2 + R^2 |f(u)|^2  \Big) \dd z\\
& \lesssim  \frac{1}{ R^2}\int_{C_{2R}(z_0)}  |u -  (\tilde{u})_{x_0, R} |^2 \dd z + R^{d+4} \|f(u)\|_{L^{\infty}(\Cc_{2R}(z_0))}^2.  \\
\end{split}
 \end{align}

To handle $\rom{3}$ we notice that by the definition \eqref{weighted_averages}, we obtain
\begin{align}
\partial_t  (\tilde{u})_{x_0, R}  = \frac{\int_{B_{2R}(x_0)}  \chi_R^2 \partial_t u \dd x}{\int_{B_{2R}(x_0)} \chi_R^2 \dd x} = -  \frac{\int_{B_{2R}(x_0)} \left(2 \chi_{R} [  \textrm{Id} \otimes \nabla \chi_{R}] : A(u) \nabla u   -  \chi_R^2 \cdot f(u) \right)\dd x}{\int_{B_{2R}(x_0)} \chi_R^2 \dd x},
\end{align}
where we have dropped the dependence of $\chi_{x_0,R}$ on $x_0$ for brevity. Notice that in the above identity there are no boundary terms thanks to our use of the weighted average $ (\tilde{u})_{x_0, R} $. Using \hyperlink{C2}{(\textbf{C2})} and the definition of $\eta$ along with $|\nabla \chi_R| \lesssim \frac{1}{R}$, we are then able to write 
\begin{align}
 |\rom{3}| & \lesssim R^{-d}  \int_{C_{2R}(z_0)}\Big(  \eta^2 | h''_\epsilon(  (\tilde{u})_{x_0, R} ) |  \, |u -  (\tilde{u})_{x_0, R} | \label{final_cacc_2} \\
&\qquad \qquad \qquad \qquad  \times  \Big( \int_{B_{2R}(x_0)} |2 \chi_{R}  [  \textrm{Id} \otimes \nabla \chi_{R}] : A(u) \nabla u |\dd x +   \int_{B_{2R}(x_0)} \chi_R^2 |f(u)| \dd x \Big)     \Big) \dd z  \\
& \lesssim  \frac{C(\gamma)}{ R^2} \int_{C_{2R}(z_0)} |u -  (\tilde{u})_{x_0, R} |^2 \dd z\\
& \,+   \frac{ \gamma R^2 }{R^{2d}} \int_{C_{2R}(z_0)} \Big( \eta^2 \Big( \int_{B_{2R}(x_0)} |2 \chi_{R}  [ \textrm{Id}\otimes \nabla \chi_{R}]: A(u) \nabla u |\dd x \Big)^2  + R^{2d} \|f(u)\|^2_{L^{\infty}(\Cc_{2R}(z_0))}\Big) \dd z\\
& \lesssim  \frac{C(\gamma)}{R^2} \int_{C_{2R}(z_0)} |u -  (\tilde{u})_{x_0, R} |^2 \dd z\\
& \,  +  \frac{ \gamma  }{R^{2d}}  \int_{C_{2R}(z_0)} \Big( \chi_R^2   R^d \int_{B_{2R}(x_0)} \tau^2 \chi_{R}^2  |\nabla u |^2 \dd x +R^{2d+2} \|f(u)\|_{L^{\infty}(\Cc_{2R}(z_0))}^2  \Big) \dd z\\
& \lesssim  \frac{C(\gamma)}{ R^2 } \int_{C_{2R}(z_0)}   |u -  (\tilde{u})_{x_0, R} |^2 \dd z+ \gamma  \int_{C_{2R}(z_0)} \eta^2 |\nabla u|^2 \dd z  +   \gamma R^{d+4} \|f(u)\|_{L^{\infty}(\Cc_{2R}(z_0))}^2,
\end{align}
for any $\gamma>0$. 

To finish, we remark that by \hyperlink{C2}{(\textbf{C2})} in the form \eqref{compare}, $\rom{2}$ in \eqref{time_deriv} can be estimated as
	\begin{align}
	\label{final_cacc_3}
	|\rom{2}|\lesssim \int_{C_{2R}(z_0)}\eta|\partial_t\eta|| u- (\tilde{u})_{x_0, R} |^2 \dd z \lesssim \frac{1}{R^2} \int_{C_{2R}(z_0)}  | u- (\tilde{u})_{x_0, R} |^2 \dd z.
	\end{align}
	
We then combine the estimates \eqref{time_deriv}, \eqref{final_cacc_june_2}, \eqref{final_cacc_june}, \eqref{final_cacc_1}, \eqref{final_cacc_1_new}, \eqref{final_cacc_2}, and \eqref{final_cacc_3} and choose $\gamma>0$ small enough to absorb the appropriate terms. Using that $\eta \equiv 1$ on $C_{R}(z_0)$ then yields the result. 
\end{proof}

\subsection{Proof of Lemma \ref{time_reg}: Another estimate of Poincar\'{e}-Wirtinger type for solutions of (\ref{cross_diffusion_system})}  
The strategy for obtaining \eqref{Poincare_new} is similar to that used in the proof of Lemma \ref{nonlinear_cacc}.
\begin{proof}
Let $t_1 \in \Gamma_{R}(t_0)$ and set $\eta =  \chi_{x_0,R}  \tau \mathds{1}_{t < t_1}$. Then, by \eqref{time_deriv}  and the bounds contained in the proof of Lemma \ref{nonlinear_cacc}, taking the time derivative of $\int_{B_{2R}(x_0)}h_{\epsilon}(u \, \vert\,(\tilde{u})_{x_0,R}) \dd x$ yields 
\begin{align}
\begin{split}
\int_{B_{2R}(x_0)}h_\epsilon(u\, \vert\,(\tilde{u})_{x_0, R}) \chi_{R}^2 \dd x \bigg|_{t=t_1}\lesssim \frac{1 }{ R^2} \int_{C_{2R}(z_0)}  |u - (\tilde{u})_{x_0, R}|^2 \dd z + R^{d+4}\|f (u)\|^2_{L^{\infty}(\Cc_{2R}(z_0))}.
\end{split}
\end{align}
Then, by the definition of $\chi_{R}$ and using a slight modification of the classical Poincar\'{e}-Wirtinger inequality along with \eqref{compare}, we obtain 
\begin{align*}
\int_{B_{R}(x_0)} | u -  (\tilde{u})_{x_0, R}|^2 \dd x \bigg|_{t=t_1}  \lesssim \int_{C_{2R}(z_0)}  |\nabla u|^2 \dd z + R^{d+4}\|f(u)\|^2_{L^{\infty}(\Cc_{2R}(z_0))}.
\end{align*}
Taking the supremum over $t_1 \in \Gamma_{R}(t_0)$ yields that 
\begin{align}
\label{higher_integrability_1}
\sup_{t \in \Gamma_{R}(t_0)} \int_{B_{R}(x_0)} | u -  (\tilde{u})_{x_0, R}|^2 \dd x \lesssim \int_{C_{2R}(z_0)}  |\nabla u|^2 \dd z +R^{d+4}\|f(u)\|^2_{L^{\infty}(\Cc_{2R}(z_0))}.
\end{align}
\end{proof}

\subsection{Proof of Proposition \ref{reverse_holder}} Given the estimates contained in Lemmas \ref{nonlinear_cacc} and  \ref{time_reg}, the proof of the reverse H\"{o}lder inequality in Proposition \ref{reverse_holder} is now a slight modification of the argument for \cite[Theorem 2.1]{GS_82}. We give the argument for $d \geq 3$ --the argument for $d = 2$ goes in a similar way.

\begin{proof}[Proof of Proposition \ref{reverse_holder}] 
The main tool that we use is Proposition \ref{technical}. Let $z_0 =0$ and $Q = \Cc_{3/2}(0)$, which we assume for simplicity is in $\Lambda$. We will show that for any $0< R < 3/2$ and $z_0^{\prime} \in Q$ such that $\Cc_{4R}(z_0^{\prime}) \subset Q$ the estimate
\begin{align}
\label{intermediate_0}
\fint_{\Cc_{R}(z^{\prime}_0)} |\nabla u|^2  \dd z  \lesssim C(\gamma)\Big\{ \Big(  \fint_{\Cc_{4R}(z^{\prime}_0)} |\nabla u |^{2_*} \dd z \Big)^{\frac{2}{2_{*}}} +   \|f(u)\|^2_{L^{\infty}(\Cc_{3/2}(0))}  \Big\}+ \gamma \fint_{\Cc_{4R}(z^{\prime}_0)} |\nabla u|^2 \dd z ,
\end{align}
holds for any $\gamma >0 $ with $2_* = 2d/(d+2)$. 

Applying Proposition \ref{technical} with $g = |\nabla u |^{2_*}$, $h =   \|f(u)\|_{L^{\infty}(\Cc_{3/2}(0))}^{2_*}$, and $q = 2/ 2_*$ yields 
\begin{align}
\label{intermediate_1}
\left( \fint_{\Cc_{1/4}(0)} |\nabla u|^{2_* p} \dd z \right)^{\frac{1}{p}} \lesssim \left( \fint_{\Cc_{1}(0)} |\nabla u|^2 \dd z \right)^{\frac{2_*}{2}} +   \|f(u)\|^{2_*}_{L^{\infty}(\Cc_{3/2}(0))}
\end{align}
 for $ p \in [ 2/2_* , 2/2_* + \delta)$ with $\delta  >0$. This gives the reverse H\"{o}lder inequality \eqref{reverse_holder_relation} for $z_0 = 0 $ and $R = 1/4$. We obtain \eqref{reverse_holder_relation} for any $z_0 \in \Lambda$ and $R >0$ by applying \eqref{intermediate_1} to  the rescaled and translated $\tilde{u}(x,t) = u((4R)^2(t - t_0), 4R (x- x_0))$, which solves \eqref{cross_diffusion_system} with the reaction terms $\tilde{f}_i (x,t)= (4R)^2 f_i(t-t_0, x-x_0)$.

It remains to show \eqref{intermediate_0}. We begin with applications of H\"{o}lder's inequality in both space and time and an application of Lemma \ref{time_reg}:
\begin{align}
\label{May_13_1}
\begin{split}
& \int_{\Cc_{2R}(z^{\prime}_0)} |u - (\tilde{u})_{x_0^{\prime}, 2R}|^2 \dd z  \\
 & \leq  \sup_{z \in \Gamma_{2R}(t^{\prime}_0) } \Big( \int_{B_{2R}(x^{\prime}_0)} |u - (\tilde{u})_{x_0^{\prime}, 2R}|^2 \dd x  \Big)^{\frac{1}{2}} \int_{\Gamma_{2R}(t^{\prime}_0)}    \Big( \int_{B_{2R}(x^{\prime}_0)} |u - (\tilde{u})_{x_0^{\prime},2R}|^2 \dd x \Big)^{\frac{1}{2}}  \dd t \\
& \lesssim  \Big( \Big(\int_{\Cc_{4R}(z^{\prime}_0)} |\nabla u|^2 \dd z \Big)^{\frac{1}{2}} + R^{\frac{d+4}{2}}   \| f(u)\|_{L^{\infty}(\Cc_{4R}(z^{\prime}_0))} \Big)\\
& \qquad \times  \int_{\Gamma_{2R}(t^{\prime}_0)}   \Big( \int_{B_{2R}(x^{\prime}_0)} |u - (\tilde{u})_{x_0^{\prime},2R}|^{2} \dd x \Big)^{ \frac{1}{4} }  \Big(  \int_{B_{2R}(x^{\prime}_0)} |u - (\tilde{u})_{x_0^{\prime},2R}|^{2_*} \dd x \Big)^{\frac{1}{2}\frac{1}{ 2_*}} \dd t,
\end{split}
\end{align}
where $2^* = 2d/ (d-2)$. Using slight variants of the classical Poincar\'{e}-Wirtinger and the Poincar\'{e}-Sobolev inequalities, we further bound the right-hand side of the last string of inequalities by:
\begin{align}
\label{May_13_2}
\begin{split}
& c R^{\frac{1}{2}}\Big( \Big(\int_{\Cc_{4R}(z^{\prime}_0)} |\nabla u|^2 \dd z \Big)^{\frac{1}{2}} + R^{\frac{d+4}{2}} \| f(u)\|_{L^{\infty}(\Cc_{4R}(z^{\prime}_0))} \Big)\\
&\qquad \qquad \times   \int_{\Gamma_{2R}(t_0^{\prime})}   \Big( \int_{B_{2R}(x^{\prime}_0)} |\nabla u |^{2} \dd x \Big)^{ \frac{1}{4}}  \Big(  \int_{B_{2R}(x^{\prime}_0)} |\nabla u |^{2_*} \dd x \Big)^{\frac{1}{2 } \frac{1}{2_*}} \dd t\\
& \lesssim R^{\frac{1}{2}}\Big( \Big(\int_{\Cc_{4R}(z^{\prime}_0)} |\nabla u|^2 \dd z \Big)^{\frac{1}{2}} + R^{\frac{d+4}{2}} \| f(u)\|_{L^{\infty}(\Cc_{4R}(z^{\prime}_0))} \Big) \\
& \qquad \qquad  \times\Big(  \int_{\Cc_{2R}(z^{\prime}_0)} |\nabla u |^{2_*} \dd z \Big)^{\frac{1}{2 } \frac{1}{2_*}}   \Big( \int_{\Gamma_{2R}(t_0^{\prime})} \Big( \int_{B_{2R}(x^{\prime}_0)} |\nabla u |^{2} \dd x \Big)^{ \frac{1}{2} \frac{2_*}{2 \cdot 2_* -1}}  \dd t \Big)^{\frac{2 \cdot  2_* -1}{2 \cdot 2_*}}\\
&\lesssim R^{\frac{3}{2} - \frac{1}{d}} \Big( \int_{\Cc_{4R}(z^{\prime}_0)} |\nabla u|^2 \dd z\Big)^{\frac{3}{4}} \Big(  \int_{\Cc_{2R}(z^{\prime}_0)} |\nabla u |^{2_*} \dd z \Big)^{\frac{1}{2} \frac{1}{2_*}}\\
& \qquad  \qquad + R^{\frac{d+7}{2}  - \frac{1}{d} } \| f(u)\|_{L^{\infty}(\Cc_{4R}(z^{\prime}_0))} \Big( \int_{\Cc_{4R}(z^{\prime}_0)} |\nabla u|^2 \dd z\Big)^{\frac{1}{4}} \Big(  \int_{\Cc_{2R}(z^{\prime}_0)} |\nabla u |^{2_*} \dd z \Big)^{\frac{1}{2 } \frac{1}{2_*}}
\end{split}
\end{align}
where $c \in \R$. Notice that in the second line above we have applied H\"{o}lder's inequality in time and in the third line we have applied Jensen's inequality for concave functions as
\begin{align*}
\Big(\int_{\Gamma(t_0^{\prime}, 2R)}  \Big( \int_{B_{2R}(x^{\prime}_0)} |\nabla u |^{2} \dd x \Big)^{ \frac{a}{2} }  \dd t \Big)^{\frac{1}{2a}} \lesssim R^{\frac{1}{a} - \frac{1}{2}}  \Big( \int_{\Cc_{2R}(z^{\prime}_0)} |\nabla u |^{2} \dd z \Big)^{ \frac{1}{4} }
\end{align*}
for $a = \frac{2_*}{2 \cdot 2_* -1}$. Treating the two terms on the right-hand side of \eqref{May_13_2} separately, we notice that an application of Young's inequality yields that 
\begin{align}
\label{May_13_3}
\begin{split}
& R^{\frac{3}{2} - \frac{1}{d}} \Big( \int_{\Cc_{4R}(z^{\prime}_0)} |\nabla u|^2 \dd z\Big)^{\frac{3}{4}} \Big(  \int_{\Cc_{2R}(z^{\prime}_0)} |\nabla u |^{2_*} \dd z \Big)^{\frac{1}{2 } \frac{1}{2_*}} \\
& \qquad \qquad \lesssim \gamma R^2 \int_{\Cc_{4R}(z^{\prime}_0)} |\nabla u|^2 \dd z + C(\gamma)  R^{-\frac{4}{d}} \Big(  \int_{\Cc_{2R}(z^{\prime}_0)} |\nabla u |^{2_*} \dd z \Big)^{\frac{2}{2_*}},
\end{split}
\end{align}
for any $\gamma>0$. For the second term of \eqref{May_13_2}, we additionally use that $R \leq 3/2$ and two applications of Young's inequality to write 
\begin{align}
\label{May_13_4}
\begin{split}
& R^{\frac{d+7}{2}  - \frac{1}{d} } \| f(u)\|_{L^{\infty}(\Cc_{4R}(z^{\prime}_0))} \Big( \int_{\Cc_{4R}(z^{\prime}_0)} |\nabla u|^2 \dd z\Big)^{\frac{1}{4}} \Big(  \int_{\Cc_{2R}(z^{\prime}_0)} |\nabla u |^{2_*} \dd z \Big)^{\frac{1}{2} \frac{1}{2_*}}\\
& \lesssim    R^{\frac{3}{2}  - \frac{1}{d} }  \Big( R^{   \frac{3(d + 4)}{4} } \| f(u)\|_{L^{\infty}(\Cc_{4R}(z^{\prime}_0))}^{\frac{3}{2}} + \Big( \int_{\Cc_{4R}(z^{\prime}_0)} |\nabla u|^2 \dd z\Big)^{\frac{3}{4}} \Big) \Big(  \int_{\Cc_{2R}(z^{\prime}_0)} |\nabla u |^{2_*} \dd z \Big)^{\frac{1}{2} \frac{1}{2_*}}\\
& \lesssim \gamma R^{2} \int_{\Cc_{4R}(z_0)} |\nabla u|^2 \dd z + C(\gamma)  \Big(R^{-\frac{4}{d}} \Big(  \int_{\Cc_{2R}(z_0)} |\nabla u |^{2_*} \dd z \Big)^{\frac{2}{2_*}}  + \| f(u)\|^2_{L^{\infty}(\Cc_{4R}(z^{\prime}_0))} \Big)
\end{split}
\end{align}
for any $\gamma >0$. To obtain \eqref{intermediate_0} we then combine \eqref{May_13_1}, \eqref{May_13_2}, \eqref{May_13_3}, and \eqref{May_13_4} with the result of Lemma \ref{nonlinear_cacc}. 

\end{proof}

\section{Argument for Corollary \ref{constant_coeff}: Interior regularity estimates for solutions of (\ref{MS_system_frozen})}

\subsection{Proof of Lemma \ref{linear_cacc}: A Caccioppoli inequality for solutions of (\ref{MS_system_frozen})}

The main idea for the proof of Lemma \ref{linear_cacc}  is to linearize the methods in the argument for Lemma \ref{nonlinear_cacc}. In particular, the motivation for the below argument is that we approximate $h_{\epsilon}$ with its Taylor expansion out to second order, keeping only the convex term. This leads us to replacing the calculation \eqref{time_deriv} by instead taking the time derivative of $ (\bar{u}-b) \cdot h^{\prime \prime}_{\epsilon}( (u)_{z_0, R}) (\bar{u}-b)$ for $b \in \R^n$.

\begin{proof}[Proof of Lemma \ref{linear_cacc} ] We use essentially the same  cut-off function $\eta$ as in the proof of Lemma \ref{nonlinear_cacc}. In particular, we let $\eta = \chi_{x_0^{\prime}, r} \tau$, where $\tau \equiv 1$ on $\Gamma_{r}(t^{\prime}_0)$ and $\tau \equiv 0 $ for $t \leq t^{\prime}_0 - (2r)^2$ such that $|\partial_t \tau| \lesssim 1/r^2$.  We then take the time derivative of $(\ub - b) \cdot h^{\prime \prime }_\epsilon((u)_{z_0, R})(\ub - b) \eta^2$:
	\begin{align}
	\label{time_deriv_2}
	\begin{split}
	& \int_{B_{2r}(x^{\prime}_0)} (\ub - b) \cdot h^{\prime \prime }_\epsilon((u)_{z_0, R})(\ub - b) \eta^2 \dd x \Big|_{t=t_0^{\prime}}\\
	&  =
	\int_{C_{2r}(z^{\prime}_0)} \partial_t ((\ub - b) \cdot h^{\prime \prime }_\epsilon((u)_{z_0, R})(\ub - b) \eta^2) \dd z
	\\&  = 2 \int_{C_{2r}(z_0^{\prime})} \sum_{i=1}^n \eta^2  (e_i \cdot  h^{\prime \prime }_{\epsilon} ((u)_{z_0, R}) e_i)  (\ub_i - b_i) \partial_t  \ub_i \dd z + 2\int_{C_{2r}(z_0^{\prime})}   (\ub - b)  \cdot  h^{\prime \prime }_{\epsilon} ((u)_{z_0, R}) (\ub - b)  \eta\partial_t \eta  \dd z
	\\&=
	- 2
	\int_{C_{2r}(z_0^{\prime})}\eta^2   \nabla \ub :  h_\epsilon''((u)_{z_0, R}) A((u)_{z_0, R}) \nabla \ub \dd z \\
	&\quad -4
	\int_{C_{2r}(z_0^{\prime})}  \eta [h_\epsilon''((u)_{z_0, R})( \ub -b )\otimes \nabla \eta ]: A((u)_{z_0,R}) \nabla \ub \dd z \\
	& \quad + 2 \int_{C_{2r}(z_0^{\prime})}  \eta^2  (\ub - b)\cdot  h^{\prime \prime }_\epsilon ((u)_{z_0, R})  f(u) \dd z\\
	&\quad + 2\int_{C_{2r}(z_0^{\prime})}   (\ub - b)  \cdot  h^{\prime \prime }_{\epsilon} ((u)_{z_0, R}) (\ub - b)  \eta\partial_t \eta  \dd z
	\end{split}
	 \end{align}
Using both the properties \hyperlink{C1}{(\textbf{C1})} or \hyperlink{C1p}{(\textbf{C1$^{\prime}$})} --which may be applied since, for $i = 1, \ldots, n$, $(\partial_i \bar{u}_1, \ldots, \partial_i \bar{u}_n) \in \Xi_0$ (this follows from the construction of $\bar{u}$ in Section \ref{campanato}) and $(u)_{z_0, R} \in \Xi_1$ --   and  \hyperlink{C2}{(\textbf{C2})}, the assumption \hyperlink{H4}{(\textbf{H4})}, and the properties of $\eta$, we can then complete the argument just as in Lemma \ref{nonlinear_cacc}. 
\end{proof}

\subsection{Proof of Corollary \ref{constant_coeff}} Using standard arguments we now upgrade the Caccioppoli estimate of Lemma \ref{linear_cacc} into the required interior regularity estimates for $\bar{u}$.

\begin{proof}[Proof of Corollary \ref{constant_coeff}] We begin by showing \eqref{constant_coeff_reg_1}. Notice that we may assume $r \leq \tilde{R}/8$ as otherwise the estimates are clear. Furthermore, we initially set $\tilde{R} =1$.

For our argument, in the cylinder $\Cc_{1}(z_0^{\prime})$ we decompose $\bar{u} = w + \bar{w}$, where $\bar{w}$ solves 
\begin{align}
\label{MS_system_frozen_4}
\begin{split}
\partial_t \bar{w} - \nabla \cdot A( (u)_{z_0,R} ) \nabla \bar{w} & = 0  \, \,\, \quad \qquad \textrm{in}  \quad \Cc_{1}(z_0^{\prime}),\\
\bar{w}& = \bar{u} \quad \quad \quad  \textrm{ on} \quad \partial^P \Cc_{1}(z_0^{\prime}).
\end{split}
\end{align}
By the triangle inequality we then have that 
\begin{align}
\label{july_1}
\int_{\Cc_r(z_0^{\prime})} |\nabla \bar{u}|^2 \dd z \lesssim \int_{\Cc_r(z_0^{\prime})} |\nabla \bar{w}|^2 \dd z +  \int_{\Cc_r(z_0^{\prime})} |\nabla w |^2 \dd z.
\end{align}

To treat the first term on the right-hand side of \eqref{july_1}, we notice that an iterative application of Lemma \ref{linear_cacc} with $f \equiv 0$ yields that 
\begin{align}
\label{iterate_Cacc}
\int_{\Cc_{1/2}(z_0^{\prime})}  |\nabla^{k+1} \bar{w}|^2 \dd z \lesssim_k \int_{\Cc_{1}(z_0^{\prime})} |\nabla \bar{w}|^2 \dd z
\end{align}
for any $k \geq 0$, where we have used that \eqref{MS_system_frozen_4} has constant coefficients. The bound for terms with time derivatives is given by 
\begin{align}
\label{iterate_Cacc_2}
\int_{\Cc_{1/2}(z_0^{\prime})}  |\partial_t^l \nabla^{k+1} \bar{w}|^2 \dd z \lesssim_{k,l} \int_{\Cc_{1}(z_0^{\prime})} |\nabla \bar{w}|^2 \dd z 
\end{align}
for $l,k \geq 0$, which can easily be shown by induction on $l$ using the equation \eqref{MS_system_frozen_4}. The Sobolev embedding and \eqref{iterate_Cacc_2} then yield
\begin{align}
\label{July_4}
\begin{split}
\int_{\Cc_r(z_0^{\prime})} |\nabla \bar{w}|^2 \dd z &\lesssim r^{d+2} \sup_{y \in \Cc_{1/2}(z_0^{\prime})} |\nabla \bar{w}(y)|^2\\
& \lesssim r^{d+2} \| \nabla \bar{w} \|_{H^{\lfloor \frac{d}{2} \rfloor +1}(\Cc_{1/2}(z_0^{\prime})) } \lesssim r^{d+2} \int_{\Cc_{1}(z_0^{\prime})}  |\nabla \bar{w}|^2 \dd z.
\end{split}
\end{align}
We complete this estimate by noticing that 
\begin{align}
\label{new_J_26}
\int_{\Cc_{1}(z_0^{\prime})}  |\nabla \bar{w}|^2 \dd z \lesssim \int_{\Cc_{1}(z_0^{\prime})}  |\nabla \bar{u}|^2 \dd z + \|f(u)\|^2_{L^{\infty}(\Cc_{1}(z_0^{\prime})) },
\end{align}
 which can be seen by testing the system 
\begin{align}
\label{MS_system_frozen_4_2}
\begin{split}
\partial_t B (\bar{w} - \bar{u}) - \nabla \cdot B A( (u)_{z_0,R} ) \nabla (\bar{w} - \bar{u})& =  - B f(u)   \, \,\, \quad \qquad \textrm{in}  \quad \Cc_{1}(z_0^{\prime}),\\
\bar{w} - \bar{u}& =  0 \quad \quad \quad \qquad \quad  \,  \textrm{ on} \quad \partial^P \Cc_{1}(z_0^{\prime}),
\end{split}
\end{align}
where again $B =\sqrt{h_{\epsilon}^{\prime\prime} ( (u)_{z_0,R})}$, with $B (\bar{w} - \bar{u})$ and using the properties of the glued entropy along with the Poincar\'{e} inequality with homogeneous Dirichlet boundary data on balls. In particular, these ingredients yield that 
\begin{align}
\label{new_J_26_e}
\int_{\Cc_{1}(z_0^{\prime})}  |\nabla (\bar{w} - \bar{u})|^2 \dd z  \lesssim \|f(u)\|^2_{L^{\infty}(z_0^{\prime})}.
\end{align}

For the second term on the right-hand side of \eqref{july_1}, we notice that $w$ solves 
\begin{align}
\label{MS_system_frozen_5}
\begin{split}
\partial_t w - \nabla \cdot A( (u)_{z_0,R} ) \nabla w & = f(u)  \, \,\, \, \, \quad \qquad \textrm{in}  \quad \Cc_{1}(z_0^{\prime}),\\
w& = 0 \quad \quad \quad \qquad   \textrm{ on} \quad \partial^P \Cc_{1}(z_0^{\prime}).
\end{split}
\end{align}
Left-multiplying the system by $B =\sqrt{h_{\epsilon}^{\prime\prime} ( (u)_{z_0,R})}$ and testing with $B w$ yields 
\begin{align}
\label{July_3}
\begin{split}
\int_{\Cc_{1}(z_0^{\prime}) } |\nabla w|^2 \dd z & \lesssim \int_{\Cc_{1}(z_0^{\prime}) }\nabla w : h_{\epsilon}^{\prime\prime} ( (u)_{z_0,R})  A( (u)_{z_0,R} ) \nabla w \dd z\\
&  \lesssim \int_{\Cc_{1} (z_0^{\prime})} |w f(u)| \dd z 
\lesssim \Big(\int_{\Cc_{1} (z_0^{\prime})} |\nabla w|^2 \dd z \Big)^{\frac{1}{2}} \|f(u)\|_{L^{\infty}(\Cc_1(z_0^{\prime}))},
\end{split}
 \end{align}
 where we have again used the properties of the glued entropy and the Poincar\'{e} inequality with homogeneous Dirichlet boundary data on balls. After rescaling, the combination of \eqref{july_1}, \eqref{July_4}, \eqref{new_J_26} and \eqref{July_3} yields \eqref{constant_coeff_reg_1}.

We now show \eqref{constant_coeff_reg_2}. Again, we assume that $r \leq \tilde{R}/8$ and to begin set $\tilde{R} =1$. The triangle inequality then allows us to write
\begin{align}
\label{triangle_26}
\int_{\Cc_r(z_0^{\prime})} |\nabla \bar{u} -  (\nabla \bar{u})_{z_0^{\prime},r} |^2 \dd z & 
\lesssim \int_{\Cc_r(z_0^{\prime})} |\nabla \bar{w} -  (\nabla \bar{w})_{z_0^{\prime},r} |^2 \dd z  + \int_{\Cc_r(z_0^{\prime})} |\nabla w -  (\nabla w)_{z_0^{\prime},r} |^2 \dd z.
\end{align}
For the first term on the right-hand side we use Lemma \ref{Poincare}, in combination with \eqref{July_4} applied to $\nabla^2 \bar{w}$ and Lemma \ref{linear_cacc} to find
\begin{align} 
\label{term_1_26}
  \int_{\Cc_r(z_0^{\prime})} |\nabla \bar{w} -  (\nabla \bar{w})_{z_0^{\prime},r} |^2 \dd z &  \lesssim r^2 \int_{\Cc_{2r}(z_0^{\prime})} |\nabla^2 \bar{w}|^2 \dd z \\
&  \lesssim r^{d+4} \int_{\Cc_{1/2}(z_0^{\prime})} |\nabla^2 \bar{w}|^2 \dd z  \lesssim r^{d +4} \int_{\Cc_{1} (z_0^{\prime})} |\nabla \bar{w} - (\nabla \bar{w})_{z_0^{\prime},1} |^2 \dd z.
\end{align} 
We can continue this estimate by noticing that 
\begin{align}
\label{new_J_26_2}
 \int_{\Cc_{1} (z_0^{\prime})} |\nabla \bar{w} - (\nabla \bar{w})_{z_0^{\prime},1} |^2 \dd z 
\lesssim  \int_{\Cc_{1} (z_0^{\prime})} |\nabla \bar{u} - (\nabla \bar{u})_{z_0^{\prime},1} |^2 \dd z + \|f(u)\|^2_{L^{\infty}(\Cc_{1}(z_0^{\prime})) },
\end{align}
which can be verified via \eqref{new_J_26_e} and the triangle inequality. In particular, we write 
\begin{align*}
 \int_{\Cc_{1} (z_0^{\prime})} |\nabla \bar{w} - (\nabla \bar{w})_{z_0^{\prime},1} |^2 \dd z 
& \lesssim  \int_{\Cc_{1} (z_0^{\prime})} |\nabla \bar{u} - (\nabla \bar{u})_{z_0^{\prime},1} |^2 \dd z +  \int_{\Cc_{1} (z_0^{\prime})} |\bar{w} - \bar{u} - (\nabla (\bar{w} - \bar{u}) )_{z_0^{\prime},1}  |^2 \dd z \\
& \lesssim \int_{\Cc_{1} (z_0^{\prime})} |\nabla \bar{u} - (\nabla \bar{u})_{z_0^{\prime},1} |^2 \dd z +  \int_{\Cc_{1} (z_0^{\prime})} |\bar{w} - \bar{u} |^2 \dd z \\
& \lesssim  \int_{\Cc_{1} (z_0^{\prime})} |\nabla \bar{u} - (\nabla \bar{u})_{z_0^{\prime},1} |^2 \dd z +  \|f(u)\|^2_{L^{\infty}(\Cc_{1}(z_0^{\prime})) }.
\end{align*}

The second term on the right-hand side of \eqref{triangle_26} is easily treated using \eqref{July_3} as  
\begin{align}
\label{term_2_26}
\int_{\Cc_r(z_0^{\prime})} |\nabla w -  (\nabla w)_{z_0^{\prime},r} |^2 \dd z \lesssim \int_{\Cc_1(z_0^{\prime})} |\nabla w|^2 \dd z \lesssim \|f (u) \|^2_{L^{\infty}(\Cc_1(z_0^{\prime}))}.
\end{align}
Combining \eqref{triangle_26}, \eqref{term_1_26}, \eqref{new_J_26_2}, and \eqref{term_2_26} yields 
\begin{align}
\int_{\Cc_r(z_0^{\prime})} |\nabla \bar{u} -  (\nabla \bar{u})_{z_0^{\prime},r} |^2 \dd z \lesssim  r^{d +4} \int_{\Cc_{1} (z_0^{\prime})} |\nabla \bar{u} - (\nabla \bar{u})_{z_0^{\prime},1} |^2 \dd z +  \|f (u) \|^2_{L^{\infty}(\Cc_1(z_0^{\prime}))}, 
\end{align}
which yields \eqref{constant_coeff_reg_2} after rescaling.

\end{proof} 

\section{Proofs of Theorem \ref{Theorem_1} and Corollary \ref{Theorem_1_corollary}: Partial $C^{0,\alpha}$-regularity}

\label{Proof_Theorem_1}

\subsection{Proof of Theorem \ref{Theorem_1}}

We are now in the position to prove Theorem \ref{Theorem_1}. Since we have already described the strategy in Section \ref{campanato}, we will now only fill in the details. 

\begin{proof}  Fix $z_0 \in \Lambda_0$, where we assume that $\Lambda_0$ satisfies \eqref{set_condition_1} and the equivalent \eqref{set_condition_2},  and $R>0$ such that $\Cc_{R} (z_0) \subset \Lambda$. We may assume $0< r< R/16$ as otherwise \eqref{excess_decay} trivially holds. 

As already mentioned in Section \ref{campanato}, our strategy is to view $u$ as a perturbation of $\bar{u}$ solving \eqref{MS_system_frozen}. In particular, by the triangle inequality we obtain \eqref{triangle}. From this latter relation combined with \eqref{energy_estimate_use}, which we must still prove, and \eqref{constant_coeff_reg_1} of Corollary \ref{constant_coeff}, we obtain \eqref{triangle_2}. We rewrite \eqref{triangle_2} here for convenience:
\begin{align}
\label{triangle_2_2}
\int_{\Cc_{2r}(z_0)} |\nabla u |^2 \dd z \lesssim   \Big( \frac{r}{R} \Big)^{d+2}\int_{\Cc_{R/8}(z_0)} |\nabla u|^2 \dd z +R^{d+4} + \int_{\Cc_{R/8}(z_0)} |\nabla \bar{v}|^2 \dd z.
\end{align}
We also now prove \eqref{energy_estimate_use}. For this, we first use the triangle inequality to write
\begin{align}
\label{July_2}
\int_{\Cc_{R/8}(z_0)} |\nabla \bar{u}|^2 \dd z \lesssim \int_{\Cc_{R/8}(z_0)} |\nabla u |^2 \dd z + \int_{\Cc_{R/8}(z_0)} |\nabla \bar{v}|^2 \dd z.
\end{align}
To control the second term on the right-hand side we test \eqref{MS_system_frozen_aux} with $B \bar{v}$ and use the properties of the glued entropy to obtain
\begin{align}
\label{energy_use_proof}
\begin{split}
 \int_{\Cc_{R/8}(z_0)} |\nabla \bar{v}|^2 \dd z &  \lesssim   \int_{\Cc_{R/8}(z_0)} \nabla \bar{v} : h^{\prime\prime}_{\epsilon} ((u)_{z_0, R}) A(  (u)_{z_0,R})   \nabla \bar{v} \dd z \\
 & = - \int_{\Cc_{R/8}(z_0)}  \partial_t  |B \bar{v}|^2 \dd z - \int_{\Cc_{R/8}(z_0)}  \nabla \bar{v} :  h^{\prime\prime}_{\epsilon} ((u)_{z_0, R}) (A(  (u)_{z_0,R}) - A(u))  \nabla u   \dd z \quad  \\
& \lesssim \int_{\Cc_{R/8}(z_0)} \big(   \gamma  | \nabla \bar{v}|^2 + C(\gamma) |\nabla u|^2 \big)  \dd z,
\end{split}
\end{align}
for any $\gamma>0 $. Notice that here we have used that $\bar{v} \equiv 0$ on $\partial^P\Cc_{R/8}(z_0)$.

We now proceed as indicated in Section \ref{campanato} and derive \eqref{error_estimate_frozen}. In particular, we repeat the calculation in \eqref{energy_use_proof} above, but treat the right-hand term using Young's and H\"{o}lder's inequalities as 
\begin{align}
\label{final_3}
\begin{split}
&- \int_{\Cc_{R/8}(z_0)}  \nabla \bar{v} :  h^{\prime\prime}_{\epsilon} ((u)_{z_0, R}) (A(  (u)_{z_0,R}) - A(u))  \nabla u   \dd z  \\
& \qquad \lesssim  \gamma \int_{\Cc_{R/8}(z_0)}  |\nabla \bar{v}|^2 \dd z  + C(\gamma) \Big( \int_{\Cc_{R/8}(z_0)}   |\nabla  u|^p  \dd z  \Big)^{\frac{2}{p}} \Big(  \int_{\Cc_{R/8}(z_0)}  |A((u)_{z_0,R}) - A(u)|^{\frac{2p}{p-2}} \Big)^{\frac{p-2}{p}},
\end{split}
\end{align}
for any $\gamma>0$ and with $p> 2$ chosen such that Proposition \ref{reverse_holder} may be applied. Applying Proposition \ref{reverse_holder} and introducing the modulus of continuity $\omega$ of $A$, we continue \eqref{error_estimate_frozen} (combined with \eqref{energy_use_proof}) as
\begin{align}
\label{modulus_estimate_1}
\begin{split}
 \int_{\Cc_{R/8}(z_0)} |\nabla \bar{v}|^2 \dd z & \lesssim  \Big( \int_{\Cc_{R/2}(z_0)}   |\nabla  u|^2  \dd z  + R^{d+4} \Big) \Big(  \fint_{\Cc_{R}(z_0)}  |A((u)_{z_0,R}) - A(u)|^{\frac{2p}{p-2}} \Big)^{\frac{p-2}{p}} \\
 &  \lesssim \Big( \int_{\Cc_{R/2}(z_0)}   |\nabla  u|^2  \dd z  + R^{d+4} \Big) \Big( \fint_{\Cc_R(z_0)} \omega( |(u)_{z_0,R} - u|^2 )^{\frac{2p}{p-2}} \dd z \Big)^{\frac{p-2}{p}}\\
& \lesssim \Big( \int_{\Cc_{R/2}(z_0)}  |\nabla  u|^2 \dd z  + R^{d+4}\Big)  \Big( \fint_{\Cc_R(z_0)} \omega( | (u)_{z_0, R}- u|^2 ) \dd z\Big)^{\frac{p-2}{p}} \\
&\lesssim    \Big( \int_{\Cc_{R/2}(z_0)}  |\nabla  u|^2 \dd z  + R^{d+4}\Big)  \omega \left( \fint_{\Cc_R(z_0)} | (u)_{z_0,R} - u|^2 \dd z \right)^{\frac{p-2}{p}} .\\
\end{split}
\end{align}
Notice that we have used the existence of $c(q) \in \R_+$ such that $\omega^q \leq c(q)  \omega$, which holds since $\omega$ is bounded, and also that $\omega$ is concave.

To finish our argument we now assume that $r = \tau R$ for some $\tau \in (0,1/16)$. By the condition \eqref{set_condition_1} on $\Lambda_0$, we can choose $R_0$ small enough so that 
\begin{align}
\label{def_chi}
\chi( z_0, R):=  \omega \left( \fint_{\Cc_R(z_0)} | (u)_{z_0,R} - u|^2 \dd z \right)^{\frac{p-2}{p}} \lesssim \tau^{d + 2}
  \end{align}
holds for $R< R_0$. Combining this with \eqref{triangle_2_2} and \eqref{modulus_estimate_1}, along with the estimate \eqref{nonlinear_cacc_eq_2} and Lemma \ref{Poincare}, we then obtain that 
\begin{align}
\label{error_estimate_frozen_4_2}
\begin{split}
\int_{\Cc_{\tau R} (z_0)} | u - (u)_{z_0,\tau R}|^2\dd z& \lesssim (\tau R)^2\Big(\tau^{d+2} \int_{\Cc_{R/8}(z_0)} | \nabla u |^2 \dd z + R^{d+4} +  \int_{\Cc_{R/8}(z_0)} |\nabla \bar{v}|^2 \dd z \Big)\\
& \lesssim \tau^{d+2} (\tau R)^2  \int_{\Cc_{R/2}(z_0)} | \nabla u|^2 \dd z +R^{d+4}\\ 
&   \lesssim \tau^{d+4}  \int_{\Cc_{ R} (z_0)} | u - (u)_{z_0,R}|^2\dd z+ R^{d+4}  
\end{split}
\end{align}
where we have used that $r< R < 1$.  By \eqref{error_estimate_frozen_4_2} we can set $\tau$ small enough so that 
\begin{align}
\label{new_june}
\phi(z_0; \tau R) \lesssim \tau^{2} \phi (z_0; R)  +  \tau^{-(d+2)}R^2,
\end{align}
where $\phi$ is defined in \eqref{excess}.

Iterating \eqref{new_june} (see, e.g., \cite[Theorem 3.1]{GS_82}) then yields \eqref{excess_decay} for any $0<r<R< R_0$ and $z_0^{\prime} = z_0$. 
Since $\chi(z_0, R)$ defined in \eqref{def_chi} is continuous in $z_0$, we find that  \eqref{new_june} holds uniformly for any  $0<r\leq R< R_0$ in a neighborhood of $z_0$.  

\end{proof}

\subsection{Proof of Corollary \ref{Theorem_1_corollary}: Estimate on the Hausdorff dimension of the singular set}
Using the characterization \eqref{set_condition_2} of the singular set $\Lambda \setminus \Lambda_0$, the argument for Corollary \ref{Theorem_1_corollary} is entirely classical and taken from \cite{GS_82}. We include it here only for completeness. 

The proof of Corollary \ref{Theorem_1_corollary} mainly depends on the following result from \cite{GG_78}.

\begin{prop}\label{giusti} For $f \in L^1_{\rm{loc}}(\Lambda)$ and $0<k<d+2$, we denote 
\begin{align}
F_k := \Big\{ z_0 \in \Lambda \, \vert \, \limsup_{\rho \rightarrow 0} \rho^{-k} \int_{\Cc_{\rho}(z_0)} |f | \dd z > 0 \Big\}.
\end{align}
Then we find that 
\begin{align*}
H^{k}(F_k) = 0.
\end{align*}
\end{prop}
\noindent As previously mentioned, we use the Hausdorff measure with respect to $\delta$ --see \eqref{hausdorff_measure}.

We now apply this result to prove Corollary \ref{Theorem_1_corollary}:
 
 \begin{proof}[Proof of Corollary \ref{Theorem_1_corollary}]
To obtain the result we apply Proposition \ref{giusti} with $f = |\nabla u|^p$, where $p>2$ is chosen such that Proposition \ref{reverse_holder} holds. For this we first notice that $ |\nabla u|^p \in L_{\loc}^1(\Lambda)$, since $ u \in L^2(0,T; H^1(\Omega; \R^n))$ and by an application of Proposition \ref{reverse_holder}. Furthermore, notice that by H\"{o}lder's inequality and the characterization  \eqref{set_condition_2} of the singular set, we have that 
\begin{align}
0<\epsilon_1<  \limsup_{\rho \rightarrow 0} \rho^{-d} \int_{\Cc_\rho(z_0)} |\nabla u|^2 \dd z  \lesssim  \limsup_{\rho \rightarrow 0}  \Big( \rho^{-(d  - (p -2))} \int_{\Cc_{\rho}(z_0)} |\nabla u|^p \dd z \Big)^{\frac{2}{p}}
\end{align}
for any $z_0 \in \Lambda \setminus \Lambda_0$. By Proposition \ref{giusti} this yields \eqref{hausdorff_measure_result} for $\gamma = p-2>0$. 
 \end{proof}
 
 \section{Proof of Theorem \ref{higher_reg}: Partial $C^{1,\alpha}$-regularity}
 \label{theorem_7_proof}

The argument for Theorem \ref{higher_reg} is a slight variation of the proof of Theorem \ref{Theorem_1}, which takes advantage of the H\"{o}lder continuity of the coefficients $A_{ij}$ and uses \eqref{constant_coeff_reg_2} as opposed to \eqref{constant_coeff_reg_1}. The H\"{o}lder regularity of the coefficients, in particular, gives us more control over the modulus of continuity called $\omega$ in the proof of Theorem \ref{Theorem_1}. The following argument is inspired by the proof of \cite[Theorem 3.2]{GS_82}. 

\begin{proof} Let $z_0 \in \Lambda_0$ and assume that $\Cc_R(z_0) \subset \Lambda$. As in the argument for Theorem \ref{Theorem_1}, we assume that $r < R/16$. 

We begin by using \eqref{constant_coeff_reg_2}, the triangle inequality, and the observation that
\begin{align}
\label{July_5}
\begin{split}
& \int_{\Cc_{R/8}(z_0)} |\nabla \bar{u} - (\nabla \bar{u})_{z_0, R/8}|^2 \dd z \leq \int_{\Cc_{R/8}(z_0)} |\nabla u - (\nabla u)_{z_0, R/8}  |^2 \dd z + \int_{\Cc_{R/8}(z_0)} |\nabla \bar{v}|^2 \dd z,
 \end{split}
\end{align}
in order to write
\begin{align}
\label{triangle_2_2_2}
\begin{split}
\int_{\Cc_r(z_0)} |\nabla u -  (\nabla u )_{z_0, r} |^2 \dd z & \leq  \int_{\Cc_r(z_0)} |\nabla u -  (\nabla \bar{u} )_{z_0, r} |^2 \dd z \\
& \lesssim   \Big( \frac{r}{R} \Big)^{d+4}\int_{\Cc_{R}(z_0)} |\nabla u - (\nabla u)_{z_0,R}|^2 \dd z + R^{d+4} + \int_{\Cc_{R/8}(z_0)} |\nabla \bar{v}|^2 \dd z.
\end{split}
\end{align}

Just as in the proof of Theorem \ref{Theorem_1}, it now only remains to treat the term involving $\nabla \bar{v}$ on the right-hand side of \eqref{triangle_2_2_2}.  For this we use the same calculation as in \eqref{modulus_estimate_1}, but additionally that for any $ s \in \R_+$ we have $\omega(s) \leq c s^{\sigma}$ with $c \in \R_+$. Also, since the conditions of Theorem \ref{Theorem_1} are satisfied, for any $\alpha \in (0,1)$ we have that 
\begin{align*}
R^{-d} \int_{\Cc_R(z_0)} |\nabla u |^2 \dd z \lesssim R^{2\alpha}
\end{align*}
holds uniformly in a neighborhood of $z_0$. Combining these observations we find that 
\begin{align}
\label{modulus_estimate_1_2}
\begin{split}
  \int_{\Cc_{R/8}(z_0)} |\nabla \bar{v}|^2 \dd z & \lesssim    \Big( \int_{\Cc_{R/2}(z_0)}  |\nabla  u|^2 \dd z  + R^{d+4}\Big)  \omega \left( \fint_{\Cc_{R/2}(z_0)} | (u)_{z_0,R} - u|^2 \dd z \right)^{\frac{p-2}{p}} \\
  &\lesssim R^{d + 2\alpha}  R^{2\alpha \sigma \frac{p-2}{p} },
\end{split}
\end{align}
where we have additionally used Lemma \ref{Poincare} and that $R <1$. 

Together, \eqref{triangle_2_2_2} and \eqref{modulus_estimate_1_2} give that 
\begin{align}
\label{triangle_final_1}
\begin{split}
&\int_{\Cc_r(z_0)} |\nabla u -  (\nabla u )_{z_0, r} |^2 \dd z \lesssim   \Big( \frac{r}{R} \Big)^{d+4}\int_{\Cc_{R}(z_0)} |\nabla u - (\nabla u)_{z_0,R}|^2 \dd z + R^{d+4} +R^{d + 2\alpha + 2\alpha \sigma \frac{p-2}{p} }.
\end{split}
\end{align}
Using the same arguments as in Theorem \ref{Theorem_1}, choosing $\alpha$ close to 1, we see that this implies that $\nabla u$ is bounded in compact subsets of $\Lambda_0$.  

Using this boundedness, we revise the estimate used to handle the second term on the right-hand side of \eqref{triangle_2_2_2}. In particular, we replace \eqref{final_3} with 
\begin{align}
\label{final_3_2}
&- \int_{\Cc_{R/4}(z_0)}  \nabla \bar{v} :  h^{\prime\prime}_{\epsilon} ((u)_{z_0, R}) (A(  (u)_{z_0,R}) - A(u))  \nabla u   \dd z  \\
&  \lesssim  \gamma \int_{\Cc_{R/4}(z_0)}  |\nabla \bar{v}|^2 \dd z  + C(\gamma) \int_{\Cc_{R/4}(z_0)}   |\nabla  u|^2  \dd z \sup_{z_0^{\prime} \in \Cc_R(z_0)} | h^{\prime\prime}_{\epsilon} ((u)_{z_0, R}) (A(  (u)_{z_0,R}) - A(u(z_0^{\prime}))) |^2\\
& \lesssim \gamma \int_{\Cc_{R/4}(z_0)}  |\nabla \bar{v}|^2 \dd z  + C(\gamma) R^{d+2 + 2 \sigma},
\end{align}
which again holds uniformly in a neighborhood of $z_0$. Notice that here we have used a version of \eqref{Aug_3_1}. Combining this with \eqref{triangle_2_2_2} yields the desired result via the same arguments as in Theorem \ref{Theorem_1}.

\end{proof}

\section{Acknowledgements}
\selectlanguage{czech}
First, we would like to thank Angsar J\"ungel for helpful discussions concerning entropy methods and cross-diffusion systems; and, for making us aware of relevant literature. We would also like to thank Miroslav Bul\'{i}\v{c}ek for helpful comments during his visits to Vienna and also during the third author's time as a post-doc at the Charles University and the first and second authors' visits there. 


  \bibliographystyle{plain}
  \bibliography{regularity_cd}

\end{document}